\documentclass[a4paper,11pt]{amsart}
\usepackage{amssymb}
\usepackage{amscd}
\usepackage{paralist}

\let\Cal\mathcal
\let\frak\mathfrak
\let\Bbb\mathbb

\newtheorem*{thm*}{Theorem \thesubsection}
\newtheorem*{lemma*}{Lemma \thesubsection}
\newtheorem*{prop*}{Proposition \thesubsection}
\newtheorem*{cor*}{Corollary \thesubsection}
\theoremstyle{definition}

\newtheorem*{definition*}{Definition \thesubsection}
\theoremstyle{remark}

\newtheorem*{example*}{Example \thesubsection}
\newtheorem*{remark*}{Remark \thesubsection}

\title[Holonomy of Cartan geometries]{Holonomy reductions of Cartan
  geometries\\ 
  and curved orbit decompositions}
\author{A.\ \v Cap, A.R.\ Gover, M.\ Hammerl}
\address{A.\v C.: Faculty of Mathematics\\
University of Vienna\\
Oskar--Morgenstern--Platz 1\\
1090 Wien\\
Austria\\
\newline A.R.G.: Department of Mathematics\\
  The University of Auckland\\
  Private Bag 92019\\
  Auckland 1142\\
  New Zealand;
Mathematical Sciences Institute\\
Australian National University \\ ACT 0200, Australia
\newline
M.H.: Department of Mathematics and Computer Science \\
Ernst Moritz Arndt University of Greifswald \\
Walther-Rathenau-Stra{\ss}e 47 \\
17487 Greifswald \\
Germany 
} 

\email{Andreas.Cap@univie.ac.at}
\email{r.gover@auckland.ac.nz}
\email{matthiasrh@gmail.com}

\def\mb{\mathbf}

\def\mr{\mathrm}
\def\mc{\mathcal}

\newcommand{\hhh}{g}

\newlength{\equwidth}
\DeclareMathOperator{\Ric}{Ric}

\DeclareMathOperator{\tr}{tr}
\DeclareMathOperator{\id}{id}

\def\X{\mathfrak X}
\def\al{\alpha}

\def\ga{\gamma}

\def\ze{\zeta}

\def\th{\theta}

\def\ka{\kappa}

\def\si{\sigma}

\def\Ups{\Upsilon}
\def\ph{\varphi}

\def\om{\omega}
\def\Ga{\Gamma}

\def\Th{\Theta}
\def\La{\Lambda}

\def\Om{\Omega}

\def\x{\times}

\def\t{\otimes}

\def\goesto{\rightarrow}

\def\embed{\hookrightarrow}

\def\na{\mathrm{\nabla}}

\def\G{\mathcal{G}}

\def\Ad{\mathrm{Ad}}
\def\Fl{\mathrm{Fl}}

\def\ddtz{{\frac{d}{dt}}_{|t=0}}
\def\ddt{\frac{d}{dt}}

\def\rr{\ensuremath{\mathbb{R}}}

\def\HH{\ensuremath{\mathcal{H}}}

\def\SL{\ensuremath{\mathrm{SL}}}

\def\SO{\ensuremath{\mathrm{SO}}}

\def\PO{\ensuremath{\mathrm{PO}}}

\def\o{\circ}
\def\sl{\ensuremath{\mathfrak{sl}}}

\def\h{\ensuremath{\mathfrak{h}}}

\def\g{\ensuremath{\mathfrak{g}}}

\def\p{\ensuremath{\mathfrak{p}}}

\def\ce{\ensuremath{\mathcal{E}}}

\DeclareMathOperator{\Hol}{Hol}

\def\sideremark#1{\ifvmode\leavevmode\fi\vadjust{\vbox to0pt{\vss
 \hbox to 0pt{\hskip\hsize\hskip1em
 \vbox{\hsize3cm\tiny\raggedright\pretolerance10000
 \noindent #1\hfill}\hss}\vbox to8pt{\vfil}\vss}}}%
                        
\def\PO{P\backslash\mc{O}}

\keywords{Cartan geometries, holonomy, overdetermined PDE, tractor
  calculus, parabolic geometry}

\subjclass[2000]{53C29; 53B15; 53A30; 53A20; 58D19; 58J70; 32V05}

\begin{document}
\maketitle

\begin{abstract}
  We develop a holonomy reduction procedure for general Cartan
  geometries. We show that, given a reduction of holonomy, the
  underlying manifold naturally decomposes into a disjoint union of
  initial submanifolds.  Each such submanifold corresponds to an orbit
  of the holonomy group on the modelling homogeneous space and carries
  a canonical induced Cartan geometry. The result can therefore be
  understood as a `curved orbit decomposition'.  The theory is then
  applied to the study of several invariant overdetermined
  differential equations in projective, conformal and
  CR-geometry. This makes use of an equivalent description of
  solutions to these equations as parallel sections of a tractor
  bundle. In projective geometry we study a third order differential
  equation that governs the existence of a compatible Einstein metric,
  and in conformal geometry we discuss almost Einstein
  scales. Finally, we discuss analogs of the two latter equations in 
  CR-geometry, which leads to invariant equations  that govern the
  existence of a compatible K\"{a}hler-Einstein metric.
  \end{abstract}

\section{Introduction}

In differential geometry great gains can be achieved when apparently
unrelated structures are shown to be concretely linked. Well known
examples include: the {\em Fefferman metric} \cite{fefferman} which
associates to a hypersurface--type CR geometry a uniquely determined
conformal geometry in one higher dimension, and {\em Poincar\'e-Einstein
  geometries} \cite{fefferman-graham}, which realise a conformal
manifold as the boundary (at infinity in a suitable sense) of a
negative Einstein Riemannian manifold. Both structures have been the
focus of considerable attention,
cf. e.g. \cite{GL,lee-ah,fefferman-graham-ambient,guillarmou-qing} and
\cite{lee-fefferman,graham-sparlingchar,baum-cr,nurowski-sparling-cr,cap-gover-cr}.
More recently, there has been excitement surrounding {\em Nurowski's
  conformal structures} \cite{nurowski-metric} which are conformal
geometries canonically associated to certain distributional structures
(generic 2 distributions in dimension 5). These
arise in the study of certain ordinary differential equations linked
to Cartan's 5-variable paper \cite{cartan-cinq} and Bryant's natural
construction \cite{bryant-3planes} of a conformal split signature
$(3,3)$-structure from a given generic rank 3 distribution, cf. also
\cite{leistner-nurowski-g2ambient, mrh-sag-rank2, mrh-sag-twistors}.
While the three different constructions mentioned above appear at
first sight to be unrelated, in fact they may be viewed as special
cases of a single phenomenon. Namely they each can be understood as
arising from a holonomy reduction of a certain connection. The
connection involved is not on the tangent bundle but is on a prolonged
structure, and is known as a Cartan connection.

In Riemannian geometry the study and application of holonomy reduction
has a long history which includes Cartan's classification of symmetric
spaces \cite{cartan-remarquable} and the de Rham decomposition theorem.
The classification of possible holonomy groups of Riemannian and (torsion
free) affine connections, as well as the construction of geometries
realising these groups, forms one of the cornerstones of differential
geometry \cite{berger-53,bryant-exceptional,merkulov-schwachhoefer,bryant-recent}.
In this setting the local geometric implications of a given holonomy reduction,
e.g.  the existence of a compatible complex structure for a given
Riemannian metric, can be readily read off from the nature of the
group arising since the connections involved are on the tangent
bundle. The notion of holonomy easily generalises to principal
connections on principal bundles, in which case the holonomy group
becomes a subgroup of the structure group of the principal bundle and
in this generality the Ambrose-Singer theorem \cite{ambrose-singer-53}
relates the holonomy 
group to the curvature form of the connection. In this case also, the
local geometric implications of reductions are evident.

Recently there has been considerable interest in understanding
holonomy questions for those Cartan connections arising naturally in
{\em parabolic geometries}; the latter form a broad class of
structures which includes conformal, CR, and projective geometries as
special cases.  For the Cartan connections of projective and conformal
structures, the possible holonomy groups have been studied, and partial
classifications are available
\cite{armstrong-projective2,armstrong-conformal,leitner-normal}. This
aspect follows the treatment of principal connections. In contrast
determining the geometric implications of reduced holonomy in this
setting is far more subtle, since the connection which defines the
holonomy lives on a prolonged bundle.  Prior to the present work there
has been no general approach for studying this problem holistically on
the manifold. In cases where geometric implications have been
discussed it was usually necessary to make certain non-degeneracy assumptions,
which means that typically they apply only to a dense open subset of
the original manifold. For instance, the \emph{conformal de Rham}
theorem discussed in \cite{armstrong-conformal} and
\cite{leitner-normal} yields a decomposition of a conformal structure
with decomposable holonomy on an open dense subset.  A link between
Einstein metrics and holonomy reductions dates back at least to
results of Sasaki \cite{sasaki}, has featured in the works mentioned
and also in \cite{arm-Ein}, which links Einstein metrics to parabolic
holonomy reductions in a wide range of settings. Again in these works the key
results are stated for open dense subsets of the manifold.  Other
results of Armstrong in \cite{armstrong-projective1} show that reduced
projective holonomy yields certain familiar geometric structures on
such subsets of the original manifold, like certain contact, or
complex projective structures. 
The complement of such a non-degenerate
open dense interior (when non-empty) carries geometric structure
itself, and an interesting aspect is how this relates to the ambient
structure.

A prominent case of the phenomenon just mentioned is that of
Poincar\'{e}-Einstein metrics, which, as pointed out in
\cite{gover-Prague}, correspond to a structure equivalent to a
certainly holonomy reduction of a conformal manifold with boundary.
Here the open dense subset carries an Einstein metric, and the
singularity set coincides with the boundary. This inherits a
conformal structure which is the conformal infinity of the Einstein
metric; the Poincar\'e-Einstein programme is precisely concerned with
relating the Einstein structure to the geometry of the conformal
hypersurface. This picture is slightly generalised by the notion of an
almost Einstein manifold which means simply any conformal manifold
with a similar type of holonomy reduction (meaning essentially the
Cartan holonomy group fixes a point), and a programme to study the
nature and geometry of the singularity set using the tractor calculus
associated to the Cartan connection (see Section \ref{partrac} below)
was developed in \cite{gover-PE,gover-aes}. Further examples and
related reductions are constructed and discussed in \cite{gover-L}.
More subtle examples (including a discussion of the singularity set)
and a treatment of decomposable conformal holonomy are presented in
the works \cite{leitner-collapse,leitner-multi,armstrongLeit} of
Leitner and Armstrong-Leitner.

The purpose of this article is to develop a completely general
approach to determining, everywhere on the manifold, the geometric
implications of any specific holonomy reduction of a Cartan
connection. We find that the behaviour just described for
Poincar\'e-Einstein manifolds is typical, at least of the simplest
cases. In broad terms, our results can be described as follows: Given
a manifold equipped with a Cartan connection we show that a holonomy
reduction of this connection determines a decomposition of the
underlying manifold into a disjoint union of initial submanifolds,
thus yielding a form of stratification.  Each such submanifold inherits a
canonical geometry from the original data. What this geometry is
varies according to the type of strata; we show how to determine this
and key aspects of how it relates to the ambient structure. In many
interesting cases this yields an open dense piece which is canonically
equipped with an affine connection, and a stratification of the closed
complement to this piece, which in turn is endowed with a geometric
structure that may be viewed as a ``limit'' of the structure on the
distinguished open set.  An important point here is that we treat all
possible reductions of the connection, not simply a reduction to the
minimal holonomy group. This means that our results apply
non-trivially to the homogeneous model, and we show that there the
decomposition coincides with an orbit decomposition with respect to
the group arising in the reduction.

Remarkably, many of the nice properties of the orbit decomposition on
the homogeneous model carry over to the curved cases without essential
changes, so we call the decomposition a \textit{curved orbit
  decomposition} in the general case. This was first observed in the
special case of almost Einstein scales for conformal structures in
\cite{gover-aes}. The basic tool to establish this in a general
setting is a comparison map between a curved geometry and the
homogeneous model and we develop this here. (A version of this was
introduced for projective structures in our article
\cite{CGH-projective}.) This proves that the curved orbits are always
initial submanifolds. In addition each curved orbit is then seen to
carry a natural Cartan geometry of the same type as the corresponding
orbit on the homogeneous model.

To obtain more detailed information on the geometric structure of the
curved orbits it is necessary to study the relation between the
curvature of the induced Cartan subgeometry and the curvature of the
original Cartan geometry. This is illustrated by several examples. 

Specialising to parabolic geometries there is a strong connection
between holonomy reductions, in the sense we treat here, and solutions
to invariant overdetermined linear partial differential equations.
This is one of the key motivations for our work. Parabolic geometries
are canonically equipped with such equations; they arise as the
equations of the first operator in certain invariant differential
sequences known as Bernstein-Gelfand-Gelfand (BGG) sequences
\cite{BGG-2001,BGG-Calderbank-Diemer}. These sequences have strong
links to symmetry and representation theory; on the model they resolve
finite dimension representations and in the curved setting one such
sequence controls deformation of structure and is connected with the
existence of infinitesimal automorphisms \cite{cap-infinitaut}.  As we
shall explain, certain special solutions of these ``first BGG''
equations are exactly equivalent to holonomy reductions in our sense.
Thus our results may be recast as describing the geometric
implications of the existence of such {\em normal solutions}. (The
term normal solution was coined in \cite{leitner-normal} in connection
with case of conformal Killing operator on differential forms.) For example
we can show that the zero locus of such a solution cannot have worse
singularities on curved geometries than the zero loci of solutions on
the homogeneous model. In fact related much finer data is available.
This vastly generalises the known results on the possible form of zero
sets for twistor spinors and Einstein rescalings,
cf. e.g. \cite{friedrich-conformalrelation,baum-friedrich-twistors,habermann-twizero,
  kuehnel-rademacher-twizero, leitner-twistor,gover-aes}.

We now give a brief outline of the article: The main result will be
developed in Section \ref{sec-holred}.  There we begin with a short
review of general Cartan geometries in \ref{2.1}, but refer the
interested reader to the extensive treatments of these structures that
can be found in \cite{sharpe,cap-slovak-par}. We introduce the notion
of holonomy of a general Cartan geometry in \ref{sec-holon} and then
discuss how a given holonomy reduction canonically induces a
decomposition of the manifold in \ref{P-types}. Our main theorem
\ref{2.5} describes the structure of the curved orbit decomposition
and the relations of the induced Cartan-subgeometries with the ambient
Cartan geometry. A particularly interesting area of applications is
formed by the solutions of BGG-equations and the study of their
zero-loci, and the general principles for this are outlined in
\ref{3.1}.  In Section \ref{sec-exmp} we study several concrete
BGG-equations. We begin in \ref{sec-exmp-proj} with a third order
differential equation that governs the existence of an Einstein
metric, whose Levi--Civita connection is projectively equivalent to a
given torsion-free affine connection.  In \ref{complex-projective} we
treat an equation on complex projective type structures which include
and generalise the $h$-projective structures of interest in the
literature \cite{Apostolov-Calderbank-Gauduchon,matveev}.  The
solutions of the equation involved describe (almost) K\"{a}hler
metrics on the manifold.  In Section \ref{Fefferman} we discuss how
our general holonomy reduction results can be applied to
Fefferman-type constructions.  Finally we discuss the equation
governing almost Einstein scales in conformal geometry in
\ref{almost-Einstein} and give an interesting analog of that equation
in CR-geometry, \ref{CR-almost-Einstein}.

\section{Holonomy reductions of Cartan geometries}\label{sec-holred}
\subsection{Cartan geometries}\label{2.1}
Let $G$ be a Lie group and $P\subset G$ a closed subgroup. The Lie
algebras of $G$ and $P$ will be denoted $\g$ and $\p$, respectively.
We will always assume that $P$ meets each connected component of $G$,
so the homogeneous space $G/P$ is connected. Cartan geometries of type
$(G,P)$ can be thought of as ``curved analogs'' of the $G$--homogeneous
space $G/P$.

A \textit{Cartan geometry} of type $(G,P)$ on a manifold $M$ is a
$P$-principal bundle $\G\goesto M$ endowed with a \textit{Cartan
  connection} $\om\in\Om^1(\G,\g)$. Denote the principal right action
of $g\in P$ on $\G$ by $r^g$ and the fundamental vector field
generated by $Y\in\p$ by $\ze_Y$, i.e.,
$\ze_Y(u)=\ddtz(u\cdot\exp(tY))$.  Then $\om$ being a Cartan
connection means that the following three properties hold:
  \begin{enumerate}[(C.1)]
  \item\label{cartan1} $\om_{u\cdot p}(T_u r^p
    \xi)=\Ad(p^{-1})\om_u(\xi)$\ for all $p\in P$, $u\in \G,$ and
    $\xi\in T_u\G$.
  \item\label{cartan2} $\om(\ze_Y)=Y$\ for all $Y\in\p$.
  \item\label{cartan3} $\om_u:T_u\G\goesto\g$\ is a linear isomorphism for all $u\in\G$.
  \end{enumerate}

  The homogeneous model of Cartan geometries of type $(G,P)$ is the
  bundle $G\goesto G/P$ with the left Maurer-Cartan form $\om^{MC}\in
  \Om^1(G,\g)$ as the Cartan connection. Indeed, the defining
  properties of a Cartan connection are obvious weakenings of
  properties of the Maurer--Cartan form, which make sense in the more
  general setting.

  The curvature $K\in\Om^2(\G,\g)$ of $\om$ is defined by
\begin{align*}
  K(\xi,\eta)=d\om(\xi,\eta)+[\om(\xi),\om(\eta)],
\end{align*}
which is exactly the failure of $\om$\ to satisfy the Maurer-Cartan
equation.  In particular, the homogeneous model has vanishing
curvature in this sense.  It is a basic fact of Cartan geometries that
$(\G,\om)$ has vanishing curvature if and only if it is locally
isomorphic to the homogeneous model $(G,\om^{MC})$.  It will sometimes
be useful to work with the curvature function
$\ka:\G\goesto\La^2(\g/\p)^*\t\g$ of $\om$, which is defined by
      \begin{align}\label{def-curvfct}
        \ka(u)(X,Y):=K(\om_{u}^{-1}(X),\om_{u}^{-1}(Y))
      \end{align}
      for $u\in\G$ and $X,Y\in\g$.

A Cartan geometry is called \textit{torsion--free} if its curvature
form $K$ has values in $\frak p\subset\frak g$, or equivalently if the
curvature function satisfies $\ka(u)(X,Y)\in\frak p$ for all
$u\in\Cal G$ and $X,Y\in\g$.

\subsection{Holonomies of Cartan geometries and reductions}\label{sec-holon}

The classical concept of holonomy can not be directly applied to a
Cartan connection. Since a Cartan connection $\om$ restricts to a
linear isomorphism on each tangent space, there are no non--constant
curves which are horizontal for $\om$ in the usual sense.  There is a
simple way, however, to connect to the classical concept. A Cartan
connection $\om\in\Om^1(\G,\g)$ is easily seen to extend canonically
to a $G$--principal connection $\hat\om\in\Om^1(\hat\G,\g)$ on the
$G$--principal bundle $\hat\G:=\G\x_PG$. The extension is
characterised by the fact that $i^*\hat\om=\om$, where $i:\G\to\hat\G$
is the obvious inclusion.  Hence for any point $\hat u\in\hat\G$\ the
holonomy group (based at $u$) of the $G$-principal connection
$\hat\om$\ is a subgroup $\Hol_{\hat u}(\hat\om)\subset G$. If we
forget about the choice of the point, we obtain a conjugacy class of
subgroups of $G$, which we denote by $\Hol(\hat\om)$.  We simply
define this to be the holonomy of the original Cartan connection,
i.e.~$\Hol(\om):=\Hol(\hat\om)$.

If the holonomy $\Hol(\om)$\ is not full, i.e.,
$\Hol(\hat\om)\subsetneq G$, the extended connection $\hat\om$\ can be
reduced. For every closed subgroup $H\subset G$\ that contains (any
conjugate of) the holonomy group $\Hol(\hat\om)$\ there exists a
reduction of structure group $\HH\overset{j}{\embed}\hat\G$ from
$G$\ to $H$ that preserves the connection. If $M$ is connected, then
one simply chooses a point $\hat u\in\hat\G$ such that $\Hol_{\hat
  u}(\hat\om)\subset H$ and defines $\HH$ as the set of all points
which can be written as $c(1)\cdot h$ for some $h\in H$ and some
horizontal curve $c:[0,1]\to \hat\G$ with $c(0)=\hat u$. One
immediately verifies that this is a principal $H$--subbundle such that
$\hat\om$ restricts to an $H$--principal connection on $\HH$;
formulated in terms of the embedding $j$, this says that
\begin{align*}
j^*\hat\om\in\Om^1(\HH,\h),
\end{align*}
with $\h$\ the Lie algebra of $H$.  The holonomy group $\Hol(\hat\om)$
is the smallest subgroup of $G$ to which the connection $\hat\om$\ can
be reduced.

By standard theory, a reduction $j:\HH\embed \hat\G$ can be
equivalently described as a section of the associated fibre-bundle
$\hat\G\times_G (G/H)=\hat\G/H$. The second description of this bundle
shows that for any $x\in M$ the fibre of $\HH$ over $x$ is mapped to a
single point in the fibre of $\hat\G\times_G (G/H)$; this describes
the section corresponding to $\HH$. Conversely, the preimage of a
section of $\hat\G/H$ under the natural projection $\hat\G\to\hat\G/H$
is an $H$--principal subbundle. Note further that $\hat\om$ induces a
(non--linear) connection on the associated bundle $\hat\G\times_G
(G/H)$. It is easy to see that the section corresponding to
$j:\HH\embed \hat\G$ is parallel if and only if $\hat\om$ restricts to
an $H$--principal connection on $\HH$.

For our general definition in the setting of Cartan geometries, it
will be very useful to avoid having a distinguished base point. Hence
we work with an abstract $G$--homogeneous space $\mc{O}$ rather than
with $G/H$. One can identify $\G\times_P\mc{O}$ with
$\hat\G\times_G\mc{O}$ so this bundle carries a natural (non--linear)
connection.

\begin{definition*}
  Let $(\G,\om)$ be a Cartan geometry of type $(G,P)$ and let $\mc{O}$
  be a homogeneous space of the group $G$. Then a \textit{holonomy
    reduction of $G$--type $\Cal O$} of the geometry $(\G,\om)$ is a
  parallel section of the bundle $\G\times_P\mc{O}$.
\end{definition*}

\begin{remark*}
  We note here that our holonomy reductions need not be minimal. The
  minimal holonomy reduction of $(\G,\om)$\ is of type $G/\Hol(\om)$,
  which reduces the structure group of the extended principal bundle
  connection $\hat\om$\ to $\Hol(\hat\om)$.  Whenever $(\G,\om)$\
  allows a holonomy reduction of type $\mc{O}=G/H$\ one necessarily
  has that $\Hol(\om)\subset H\subset G$.
\end{remark*}

\subsection{Parallel sections of tractor bundles and corresponding
  holonomy reductions}\label{partrac}
For a $G$-representation $V$\ the associated bundle
$\mc{V}=\hat\G\times_G V=\G\times_P V$ is a \textit{tractor bundle}
and the linear connection induced by $\hat\om$ is called its
\textit{tractor connection}.  Sections of $\mc{V}$ can be identified
with $G$--equivariant smooth maps $\hat\G\to V$, and such a map
$s:\hat\G\to V$ corresponds to a parallel section if and only if it is
constant along any curve $c:[0,1]\goesto\hat\G$ which is horizontal in the
sense that $\hat\om(\ddt(c(t)))=0$.
\begin{lemma*}\label{lem-parorbit}
  The image of the map $s:\hat\G\goesto V$ corresponding to a parallel
  section of $\mc{V}$ over a connected manifold $M$ is a $G$--orbit
  $\mc{O}\subset V$.
\end{lemma*}
\begin{proof}
  Since $M$ is assumed to be connected we can take a smooth curve
  joining two given points $x,x'\in M$ and lift it to a horizontal
  curve $c:[0,1]\goesto \hat\G$. By the usual formula for an
  associated connection, $s$ has to be constant along $c$, so
  $s(u)=s(u')\in V$, where $u:=c(0)\in\hat\G_x$ and
  $u':=c(1)\in\hat\G_{x'}$. By $G$--equivariancy $s(\hat\G_x)$
  coincides with the $G$--orbit of $s(u)$, and the same is true for
  $s(\hat\G_{x'})$.
\end{proof}

Hence any parallel section of $\mc{V}$ canonically determines a
$G$--orbit $\Cal O\subset V$, which we will refer to as its
\textit{$G$--type}. Of course, $\G\x_P\Cal O$ is then a subbundle of
$\mc{V}$ and this inclusion is compatible with the natural
connections. Thus a parallel section of $\mc{V}$\ of $G$-type $\mc{O}$
is the same as a holonomy reduction of $(\G,\om)$\ of this $G$--type.

\subsection{The canonical $P$-type decomposition induced 
by a holonomy reduction}\label{P-types}

So far our description of holonomy reductions did not take into account
that the principal $G$-bundle $\hat \G$\ is the extended bundle
of the Cartan bundle $\G$. Since $\hat\G=\G\times_P G$
is the quotient of $\G\times G$\ by the $P$-right action
\begin{align*}
  (u,g)\cdot p=(up,p^{-1}g), u\in\G,g\in G,p\in P,
\end{align*}
there is a canonical embedding $\G\embed\hat\G$ which maps
$u\in\G$\ to $(u,e)\cdot P\in\hat\G$.
This canonical $P$-subbundle of $\hat\G$\ gives rise
to a pointwise invariant that is specific to holonomy reductions
of Cartan geometries:
\begin{definition*}\label{def-decomp}
  Let $(\G\to M,\om)$ be a Cartan geometry of type $(G,P)$ together
  with a holonomy reduction of type $\mc{O}$ described by
  $s:\hat\G\goesto \mc{O}$.  Then for a point $x\in M$ the
  \textit{$P$--type of $x$ with respect to $s$} is the $P$-orbit
  $s(\G_x)\subset\mc{O}$.
\end{definition*}
  
We denote by $\PO$ the set of $P$--orbits in $\Cal O$. Then for a
Cartan geometry $(\G\to M,\om)$ endowed with a holonomy reduction of
type $\Cal O$, the base manifold $M$ decomposes into a disjoint union
according to $P$--type as
  \begin{align*}
    M=\bigcup\limits_{i\in \PO}M_i.  
  \end{align*}
We term the components $M_i$ {\em curved orbits} for a reason that
will shortly be obvious.

\subsection{The $P$-type decomposition of the homogeneous model}\label{P-type-homog}  
  We now study the $P$-type decomposition on the homogeneous model for
  a given holonomy reduction of $G/P$\ of type $\Cal O$.
  The extended bundle $\hat\G=G\times_P G$ can be
  canonically trivialised via 
  \begin{align}\label{homtriv}
    &G/P\times G\goesto G\times_P G \notag \\
    &(gP,g')\mapsto [g,g^{-1}g']_P.
  \end{align}
  Using this, the $G$-equivariant map $s:\hat\G\goesto\mc{O}$ defining
  a holonomy reduction corresponds to a smooth map $G/P\times G\goesto
  \mc{O}$, which by $G$--equivariancy is determined by its restriction
  to $(G/P)\x\{e\}$. Moreover, the extended principal connection
  $\hat\om$ in this case is just the flat connection coming from this
  trivialisation. In particular, any curve $c:[0,1]\goesto G/P\times
  G$ of the form $c(t)=(\bar c(t),e)$ is horizontal, so the map
  $s:G/P\times G\goesto \mc{O}$ has to be given by
  $s(gP,g')=(g')^{-1}\cdot\al$ for some fixed element $\al\in\Cal O$.
  Fixing a holonomy reduction of $G/P$ (of type $\Cal O$) thus amounts
  to fixing an element $\al=s(eP,e)\in\Cal O$, and we denote
  $H=G_{\al}\subset G$ the isotropy group of this element. In
  particular, we can then identify $\Cal O$ with $G/H$.
  
  To determine the $P$-type of $x=gP\in G/P$ with respect to $s$, we
  observe that in the trivialisation \eqref{homtriv} the fibre $\Cal
  G_{gP}\subset\hat\G_{gP}=\{gP\}\x G$ is just $\{gP\}\x \{gb:b\in
  P\}$.  It follows by equivariancy that $s(\G_x)$ is the $P$--orbit
  $Pg^{-1}\cdot s(eP,e)=Pg^{-1}H\in P\backslash G/H=\PO$.  The map
  $G/P\goesto \PO$ which sends each point to its $P$--type thus
  factors to a bijection  
\begin{align}\label{Modeltypemap}
  H\backslash G/P&\goesto \PO=P\backslash G/H,\\ \notag HgP&\mapsto
  Pg^{-1}H
  \end{align}
  of double coset spaces, compare with Proposition 2.13 of
  \cite{CGH-projective}.
  
  This shows that for $M=G/P$, we get $M_{Pg^{-1}H}=HgP/P=H\cdot
  (gP/P)\subset G/P$. Hence the decomposition of $G/P$ according to
  $P$--type with respect to the holonomy reduction determined by
  $\al\in\Cal O$ coincides with the decomposition of $G/P$ into orbits
  under the action of the subgroup $H=G_\al\subset G$.  Now it is
  clear that each of the $H$-orbits naturally shows up as the
  homogeneous model of a Cartan geometry. The stabiliser
  of a point $gP\in G/P$ in $H$ of course is $H\cap gPg^{-1}$ and
  hence $HgP \cong H/(H\cap gPg^{-1})\cong (g^{-1}H g)/(g^{-1}H g\cap
  P)$. 

\begin{remark*}
It looks as if there were many different holonomy reductions of type
$\Cal O$ of the homogeneous model. This is true, but they are all
related by the action of $G$. For our purposes, the main difference
between these reductions is the $P$--type of the origin $eP\in G/P$. Up
to $G$--action, there is only one holonomy reduction of type $\Cal O$
of $G/P$, whence we will talk about \textit{the} model of holonomy
reductions of type $\Cal O$.
\end{remark*}

\subsection{Curved orbit decomposition and induced Cartan
  geometries}\label{2.5}

Consider a $G$--homogeneous space $\Cal O$ and two elements
$\al,\al'\in\Cal O$. If $\al'=g\cdot\al$, then the stabilisers are
conjugate, so $G_{\al'}=gG_\al g^{-1}$. If we in addition assume that
$\al$ and $\al'$ lie in the same $P$--orbit in $\Cal O$, then we can
choose $g\in P$, and thus $gPg^{-1}=P$. Consequently, putting
$P_\al:=G_\al\cap P$ and likewise for $\al'$, we see that
$P_{\al'}=gP_\al g^{-1}$. Thus we see that $(G_\al,P_\al)$ is
isomorphic to $(G_{\al'},P_{\al'})$ as pair consisting of a group endowed with a
distinguished subgroup. In the formulation of our main result below,
given a $P$--orbit $i\in\PO$, we will denote by $(H_i,P_i)$ an
abstract representative of this isomorphism class of groups with a
distinguished subgroup.

\begin{thm*}
  Let $(\G,\om)$\ be a Cartan geometry of type $(G,P)$ which is
  endowed with a holonomy reduction of type $\mc{O}$. Consider a
  $P$--orbit $i\in \PO$ such that the curved orbit $M_i$ is
  non--empty, and consider the corresponding groups $P_i\subset H_i$
  as discussed above. Then
\begin{enumerate}[(i)]
\item \label{thm-curvedorbits-orb} 
  Choose a representative $\al\in\Cal O$ for the Orbit $i$, let
  $G_\al\in G$ be its stabiliser and consider the holonomy reduction
  of the homogeneous model $G/P$ determined by $\al$ as in Section
  \ref{P-type-homog}. Then for each $x\in M_i$ there exist
  neighbourhoods $N$ of $x$ in $M$, and $N'$ of $eP$ in $G/P$ and a
  diffeomorphism $\ph:N\goesto N'$ with $\ph(x)=eP$ and $\ph(M_i\cap
  N)=(G_\al\cdot eP)\cap N'$. In particular, $M_{i}$ is an initial
  submanifold of $M$.
\item \label{thm-curvedorbits-cart} $M_i$ carries a canonical Cartan
  geometry $(\G_i\goesto M_i,\om_i)$ of type $(H_i,P_i)$. 

  Choosing a representative $\al$ for $i\in\PO$ as in (i) and
  identifying $(H_i,P_i)$ with $(G_\al,P_\al)$, we obtain an embedding
  of principal bundles $\mb{j}_{\al}:\G_i\to\Cal G|_{M_i}$ such that
  $\mb{j}_{\al}^*\om=\om_i$. Thus $(\Cal G_i,\om_i)$ can be realised
  as a $P_{\al}$--subbundle in $\Cal G|_{M_i}$ on which $\om$ restricts to a
  Cartan connection of type $(G_\al,P_\al)$.
\item \label{thm-curvedorbits-curv} For the embedding $\mb{j}_{\al}$
  from (ii), the curvatures $K$ of $\om$ and $K_i$ of $\om_i$ are
  related as 
  \begin{align*}
    K_i=\mb{j}_{\al}^* K.
  \end{align*}
In particular, if $\om$ is torsion free, so is $\om_i$. 

  Likewise, let $\ka$ and $\ka_i$ denote, respectively, the curvature
  functions of the two Cartan connections. Then $\ka(\mb{j}_{\al}(u))$
  maps $\La^2(\frak g_\al/(\frak g_\al\cap \frak p)$ to $\frak g_\al$
  and its restriction to this subspace coincides with $\ka_i(u)$.
\end{enumerate}
\end{thm*}

  The proof of the theorem is based on the following comparison
  method.
  \begin{lemma*}[Comparison]\label{lem-comp}
    Let $(p:\G\to M,\om)$\ and $(p':\G'\to M',\om')$ be Cartan
    geometries of type $(G,P)$ that are endowed with holonomy
    reductions of type $\mc{O}$ described by $s:\hat\G\to\Cal O$,
    respectively $s':\hat\G'\to\Cal O$. Assume that for some
    $\al\in\Cal O$ with $P$--orbit $\bar\al\in\PO$ both curved orbits
    $M_{\bar\al}$ and $M'_{\bar\al}$ are non--empty.  Then for points
    $x\in M_{\bar\al}$ and $x'\in M_{\bar\al}$ we obtain:
    \begin{itemize}
    \item  A diffeomorphism $\phi:N\goesto N'$ from an open
    neighbourhood of $x$ in $M$ to an open neighbourhood of $x'$ in $M'$
    such that $\phi(x)=x'$. 
  \item A $P$-equivariant diffeomorphism
    $\Phi:p^{-1}(N)\goesto(p')^{-1}(N')$ which covers $\phi$ and
    satisfies 
  \begin{align}
      \label{Phipullback}
      s'\o\Phi = s .
    \end{align}
 \end{itemize}
    
 In particular, it follows that
  \begin{align}\label{Typemap}
    \phi(M_i\cap N)=M'_i\cap N'
  \end{align}
  for all $i\in \PO$. 
  \end{lemma*}
  We remark that some of the intersections $M_i\cap N$ may be empty.

  \begin{proof}[Proof of lemma \ref{2.5}]
    The proof is based on an adapted version of normal coordinates for
    Cartan geometries. For this we fix a linear subspace
    $\g_-\subset\g$ which is complementary to the subspace
    $\p\subset\g$. For $X\in\g_-$, we denote by $\tilde X\in\X(\G)$
    the vector field characterised by $\om(\tilde X)=X$. Next, take a
    point $u\in\G_x$ such that $s(u)=\al$, and consider the flow
    $\Fl^{\tilde X}_1(u)$ of $\tilde X$ starting in $u$ up to time 1.
    This is defined for $X$ in a neighbourhood $W$ of zero in $\g_-$
    and, possibly shrinking $W$, $\Psi(X):=\Fl_1^{\tilde X}(u)$ defines a
    smooth map $W\to \G$ such that $\psi:=p\o\Psi$ is a diffeomorphism
    from $W$ onto an open neighbourhood $N$ of $x=p(u)$ in $M$. These
    are the local normal coordinates around $x$ determined by $u$.
    
    Next, we define a local section $\tau:N\goesto\G$ by 
    $\tau(\psi(X)):=\Psi(X)$, and an adapted local section
    $\hat\tau:N\goesto\hat\G$ by $\hat\tau(\psi(X)):=\Psi(X)\cdot\exp(-X)$.

 Then $\hat\tau$ has the property that
 for fixed $X\in\g_-$ the curve $c$ defined by 
$$
c(t):=\hat\tau(\psi(tX))=\Psi(tX)\cdot\exp(-tX)=\Fl_t^X(u)\cdot\exp(-tX)
$$
(for sufficiently small $t$) is horizontal for the principal
connection $\hat\om$. Indeed, we have
$$
    \hat\om (c'(t))=\Ad(\exp(tX))X-X=0. 
    $$
    But since $s:\hat\G\goesto\mc{O}$ is constant along horizontal
    curves we conclude that for $X\in W$ we get
    $$
    \al=s(c(0))=
    s(c(1))=s(\hat\tau(X))=s(\Psi(X)\cdot\exp(-X)),
$$
  and by $G$-equivariancy, we obtain  
  \begin{equation}\label{Ptype-formula}
    s(\Psi(X))=\exp(-X)\cdot\al.
  \end{equation}
  
  Now we can perform the same construction for $(p':\G'\to M',\om')$
  and a point $u'\in\G'_{x'}$ such that $s'(u')=\al$. Shrinking the
  neighbourhoods of zero in $\g_-$ appropriately, we may assume that
  $W'=W$, and put $\phi:=\psi'\circ \psi^{-1}:N\goesto N'$, so
  $\phi(x)=x'$. Further there evidently is a unique $P$-equivariant
  diffeomorphism $\Phi:p^{-1}(N)\goesto (p')^{-1}(N')$ such that
  $\Phi\o\tau=\tau'$ and by construction, this covers $\phi$.

  Since $s'(\Psi'(X))=\exp(-X)\cdot\al=s(\Psi(X))$ one immediately
  obtains \eqref{Phipullback}.  For the last claim, recall that by
  definition $y\in M_i$ is equivalent to the fact that $s(\G_y)$ is
  the orbit $i$. For $v\in\G_y$ put $v'=\Phi(v)\in\G'_{\ph(y)}$ we
  then have by \eqref{Phipullback} that $s'(v')=s(v)$ and therefore
  $y$ and $\ph(y)$ have the same $P$--type.
\end{proof}

\begin{proof}[Proof of Theorem \ref{2.5}]
  We choose a representative $\al$ for the orbit $i\in\PO$ and apply
  Lemma \ref{lem-comp} to the case where $M'=G/P$ is the homogeneous
  model of type $(G,P)$ with the holonomy reduction determined by
  $\al\in\Cal O$, so $eP\in M'_i$. Let $\phi:N\goesto N'$,
  $\Phi:p^{-1}(N)\goesto(\pi')^{-1}(N')$ be the maps constructed in
  the lemma for the given point $x\in M_i$ and $x'=eP\in (G/P)_i$.
  
  Claim (i) immediately follows from the fact that
  $(G/P)_{\bar\al}=G_\al\cdot eP\subset G/P$ is a $G_\al$-orbit
  (recall \eqref{Modeltypemap}) and from formula \eqref{Typemap}. The fact
  that orbits are initial submanifolds is well known, see e.g.~Theorem
  5.14 in \cite{KMS}.
  
  To prove (ii), observe first that via the inclusion $j:M_i\embed M$,
  we can pull back $\G$ and $\hat\G$ to a principal $P$--bundle
  respectively a principal $G$--bundle over $M_i$. In view of the
  discussion in \ref{sec-holon}, the reduction of $\hat\G$ determined
  by $s$ can be described as the pre--image
  $\hat\G_{\al}=s^{-1}(\al)\subset\hat\G$, and we define
  $$
  \G_{\al}:={j}^{*}(\hat\G_{\al})\cap j^*\G.
  $$
  We claim that this is a principal bundle with structure group
  $P_\al=G_\al\cap P$ over $M_i$.  This is again proved by comparison
  to the homogeneous model $M'=G/P$ and its holonomy reduction
  determined by $\al$: In Lemma \ref{lem-comp} we (locally)
  constructed a $P$-bundle map $\Phi:p^{-1}(N)\goesto (p')^{-1}(N')$
  such that $s'\o\Phi=s$. Hence it clearly suffices to prove that
  $\G'_{\al}$ is a $P_\al$--principal subbundle over $M'_{\bar\al}$.
  
  But on the homogeneous model we can use the trivialisation
  $\hat\G'=G\times_P G\cong (G/P)\times G$, and there we simply have
  $\hat\G'_{\al}=(G/P)\times G_\al$.  Therefore
  $(j')^{*}(\hat\G'_{\al})=(G_\al/P_\al)\times G_\al$\ and
  $(j')^{*}(\hat\G'_{\al})\cap (j')^{*} \G'=(G_\al/P_\al)\times
  P_\al$.  In particular this shows that the intersection
  $(j')^{*}(\hat\G'_{\al})\cap (j')^{*} \G'$ is a $P_\al$-subbundle
  of the $P$-bundle $(j')^{*} \G'$ over $M'_{\bar\al}=G_\al\cdot eP$.
  Therefore also $\G_{\al}$ is an $P_{\al}$--principal subbundle of
  $j^*\G$ over $M_i$. 

We next claim that $j^*\om$ defines a Cartan connection of type
$(G_\al,P_\al)$\ on $\G_{\al}$. For this, first note that the extended
$G$-principal connection form $\hat\om\in\Om^1(\hat\G,\g)$\ has values
in $\frak g_\al$\ on $\hat\G_{\al}$, and in particular
$(\om_{\al})_u(T_u\G_{\al})\subset\frak g_\al$\ for all
$u\in\G_{\al}$.  But since $(j^*\om)_u$\ is injective on $T_u(j^*\G)$
a simple counting of dimensions shows that
$(j^*\om)_u:T_u\G_{\al}\goesto\frak g_\al$ is a linear isomorphism,
which yields (C.3).  The necessary equivariance (C.1) and reproduction
(C.2) properties follow immediately from those of $\om$ by
restriction.
  
  Now let $b\in P$ and $\al'=b\cdot\al$ be another point in
  $P\cdot\al=i\in\PO$. We know that $\G_{\al}$\ and $\G_{\al'}$ are
  principal subbundles of $\G$ with structure group $P_{\al}$ and
  $P_{\al'}$, respectively to which $\om\in\Om^1(\G,\g)$ restricts
  nicely. Then one immediately checks that the restriction of the
  principal right action $r^{b^{-1}}$ induces an isomorphism between
  the two principal subbundles (which is equivariant over the
  isomorphism $P_\al\cong P_{\al'}$ induced by
  conjugation). Equivariancy of $\om$ further implies that this
  isomorphism is compatible with the induced Cartan connections (where
  we identify $\frak g_\al$ and $\frak g_{\al'}$ via the isomorphism
  induced by $\Ad(b^{-1})$). Hence we can view the result of our
  construction as a canonical Cartan geometry $(\Cal G_i\to
  M_i,\om_i)$ together with an inclusion $\mb{j}_{\al}$ induced by the
  choice of a representative $\al$ of $i\in\PO$ as claimed in (ii),
  which completes the proof of this part.
  
  The first part of (iii) then follows immediately from the definition
  of the Cartan curvature $K=d\om+\frac{1}{2}[\om,\om]$\ of $\om$ and
  pullback via $\mb{j}_{\al}$, while the second part is just the
  obvious restatement of this in terms of the curvature functions.
\end{proof}

\subsection{Parabolic geometries and normal solutions of BGG
  equations}\label{3.1}
We now consider special holonomy reductions for Cartan geometries of
type $(G,P)$ with $G$ a semisimple Lie group and $P\subset G$ a
parabolic subgroup. In this case the Lie algebra $\g$\ has a natural
grading
$$
\g=\underbrace{\g_{-k}\oplus\cdots \oplus\g_{-1}}_{\g_-}\oplus
\underbrace{\g_0 \oplus \g_1\oplus \cdots\oplus \g_k}_{\p},
$$ such that $\g_{-1}$\ generates $\g_-$.  This class of Cartan
geometries is particularly interesting from several points of
view. First the class includes many structures already studied
intensively, such as conformal geometry, projective geometry, and
(hypersurface type) CR geometry.  Second, for every structure in the
class one has canonical regularity and normality conditions on the
Cartan connection $\om$ which lead to Cartan geometries which are
equivalent (in a categorical sense) to underlying geometric
structures. Finally via the canonical Cartan connection and related
calculus the structures in the class admit the application of
efficient tools from representation theory to geometric problems; the
relevant representation theory is far from trivial, but is very well
studied. For extensive background on this class of geometries we refer
to \cite{cap-slovak-par}.
  
For a parabolic geometry of type $(G,P)$, \cite{BGG-2001} introduced a
construction for a natural sequence of linear differential operators
that was then simplified in \cite{BGG-Calderbank-Diemer}.  For each
tractor bundle $\mc{V}=\G\times_P V$, with $V$ irreducible for $G$,
one obtains the \textit{generalised BGG-sequence}
$$
  \Ga(\mc{H}_0)\overset{\Th^V_0}{\goesto}\Ga(\mc{H}_1)\overset{\Th^V_1}{\goesto}\cdots
  \overset{\Th^V_{n-2}}{\goesto}
  \Ga(\mc{H}_{n-1})\overset{\Th^V_{n-1}}{\goesto}\Ga(\mc{H}_n).
  $$
  Here each $\mc{H}_k$ is a certain subquotient bundle of the
  bundle $\La^k T^*M\t \mc{V}$ of $\mc{V}$--valued $k$--forms, and
  each $\Th^V_i$ is a linear differential operator intrinsic to the
  given geometry.
  
  We are mainly interested in the operator $\Th^V_0$, which defines an
  overdetermined system and is closely related to the tractor
  connection $\na$ on $\mc{V}$.  The parabolic subgroup $P\subset G$
  determines a filtration on $V$ by $P$--invariant subspaces. We only
  need the largest non--trivial filtration component $V^0\subset V$.
  Then $\mc{H}_0$ is simply the quotient of $\mc{V}/\mc{V}^0$, and we
  denote by $\Pi:\Ga(\mc{V})\to\Ga(\mc{H}_0)$ the natural projection.
  
  It turns out that the bundle map $\Pi$ can be used to identify
  parallel sections of $\mc{V}$, with special solutions of the first
  BGG operator $\Th^V_0$, which are then called \textit{normal
    solutions}.  More precisely, one has:

\begin{thm*}\label{normp}
  Let $\mc{V}$ be a $G$-irreducible tractor bundle on $M$. The bundle map
  $\Pi$ induces an injection from the space of parallel sections of
  $\mathcal{V}$ to a subspace of $\Ga(\mathcal{H}_0)$ which is
  contained in the kernel of the first BGG operator
$$
\Th^V_0 : \Ga(\mc{H}_0)\to\Ga(\mc{H}_1) .
$$
\end{thm*}
\begin{proof}
If $s\in\Ga(\Cal V)$ is parallel, then since $\partial^*$ is identically 
zero on $ \Ga(\Cal V)$, part (2) of Lemma 2.7 of \cite{BGG-2001} shows that 
$s=L(\Pi(s))$, where $L$ is the so called splitting operator. This implies
injectivity and since $\Th^V_0(\Pi(s))$ by definition is a projection of 
$\nabla L(\Pi(s))$, we see that $\Pi(s)$ is in the kernel of $\Th^V_0$. 
\end{proof}

Let $U\subset V$ be a $P$--invariant subspace. Then the associated
bundle $\mc{U}$ is a subbundle of $\mc{V}$. For a normal solution
$\si$ of $\Th^V_0$ let $s\in\Ga(\mc{V})$ be the parallel section such
that $\Pi(s)=\si$. Define $\mc{Z}^U(\si):=\{x\in
M:s(x)\in\mc{U}_x\subset\mc{V}_x\}$. Of course, this is just the zero
set of the section of $\mc{V}/\mc{U}$ obtained by projecting $s$ to
the quotient. Note that for $U=V^0\subset V$, the largest proper
filtration component, we get $\mc{Z}^{V^0}(\si)=\mc{Z}(\si)$, the zero
set of $\si$. For other proper filtration components $U\subset V$ we
have $U\subset V^0$ and hence $\mc{Z}^U(\si)\subset \mc{Z}(\si)$ can
be viewed as a space of ``higher order zeros'' of $\si$. This point of
view can be made precise using the fact that $s$ can be described as
the image of $\si$ under a linear differential operator
\cite{BGG-2001,BGG-Calderbank-Diemer}. More generally, for
$P$--invariant subspaces $U\subset U'\subset V$ one has
$\mc{Z}^U(\si)\subset\mc{Z}^{U'}(\si)$. This typically yields a
stratification of the zero set of $\si$, examples of which
were given in \cite{CGH-projective}. 

The parallel section $s$ of $\mc{V}$ gives rise to a holonomy
reduction of type $\Cal O$ for some $G$--orbit $\Cal O\subset V$,
called the $G$--type of $s$. We will also refer to the orbit
$\mc{O}\subset V$ as the $G$-type of the normal solution
$\si=\Pi(s)$. According to Definition \ref{def-decomp} the holonomy
reduction provides a curved orbit decomposition
$M=\bigcup\limits_{i\in \PO}M_i$. We will also refer to the $P$--type
of $x\in M$ as the \textit{$P$--type with respect to $\si$}.  For a
$P$--invariant subspace $U\subset V$, the subspace $U\cap\Cal O$ is
$P$--invariant, so it is a union of $P$--orbits. Clearly, we have
\begin{equation}
  \label{equ-zerodecomp}
  \mc{Z}^U(\si)=\bigcup\limits_{i\in P\backslash(U\cap \mc{O})}M_i. 
\end{equation}

As for holonomy reductions, we can describe all normal solutions on
the homogeneous model $G/P$ for some given $G$--type $\mc{O}\subset
V$: For any element $v\in\mc{O}$, the $P$--equivariant function
$G\goesto V$ defined by $g\mapsto g^{-1}v$ defines a parallel section
of $\mc{V}$. Via the trivialisation $\hat G=G\times_P G\cong G/P\times
G$ it is easy to see that every parallel section of $\mc{V}$ is
obtained in that way, and it turns out that the space of parallel
sections surjects onto the kernel of $\Th^V_0$ on the homogeneous
model, i.e.~all solutions are normal in this case.

Our results on curved orbit decompositions now easily imply that
locally, all possible forms of the sets $\mc{Z}^U(\si)\subset M$
already show up on the homogeneous model $G/P$.
\begin{prop*}
  Let $\si$ be a normal solution of $\Th_0^V$ on $(\G\goesto M,\om)$
  of $G$-type $\mc{O}\subset V$ and let $x\in M$ be any point. Then
  there is a (normal) solution $\si'$ on $(G\goesto G/P,\om^{MC})$ for
  which $eP\in G/P$ has the same $P$--type with respect to $\si'$ that
  $x$ has with respect to $\si$. Further, there are open neighbourhoods
  $N$ of $x$ in $M$ an $N'$ of $eP$ in $G/P$ and there is a
  diffeomorphism $\ph:N\goesto N'$, such that $\ph(x)=\ph(x')$ and
  $\ph(\mc{Z}^U(\si)\cap N)=\mc{Z}^U(\si')\cap N'$ for any
  $P$--invariant subspace $U\subset V$.
\end{prop*}
\begin{proof}
  Consider the equivariant function $s:\Cal G\to V$ corresponding to
  the parallel section of $\mc{V}$ which induces $\si$. Choose a point
  $u\in\Cal G_x$ and put $v=s(u)\in \mc{O}$. Let $\si'$ be the normal
  solution on $G/P$ determined by the function $g\mapsto g^{-1}\cdot
  v$, the claim about $P$--types follows. Then the result follows
  immediately from Theorem \ref{2.5}, since the set $\mc{Z}^U(\si)$ is
  a union of curved orbits in $M$, while $\mc{Z}^U(\si')$ is the union
  of the corresponding orbits in $G/P$.
\end{proof}

\section{Examples and applications}\label{sec-exmp}

\subsection{Metrics on the projective standard tractor bundle}\label{sec-exmp-proj} 

Let $(M,[D])$ be an oriented smooth $n$-manifold endowed with a
projective equivalence class of torsion-free affine connections.
Hence the equivalence class of $D$ consists of all those torsion-free
affine connections which have the same geodesics as $D$ up to
pa\-ra\-me\-tri\-za\-tion.  It is well known that $D$ and $\hat D$ are
projectively equivalent if and only if there is a $1$--form $\Ups$
such that
\begin{align*}
  \hat D_a\ph_b=D_a\ph_b +\Ups_a\ph_b+\Ups_b\ph_a
\end{align*}
for all $\ph\in\Om^1(M)$, see e.g.~\cite{eastwood-notes}, also for the
notation.

An oriented projective structure can be equivalently described as a
Cartan geometry $(\G,\om)$ of type $(G,P)$, where $G=SL(n+1,\Bbb R)$
and $P\subset G$ is the stabiliser of a ray $\rr_+ X\in\rr^{n+1}$. In
particular, the homogeneous model is the projective $n$--sphere $S^n$,
which is a $2$-fold covering of projective $n$--space $\rr P^n$.  The
bundle associated to the standard representation of $\SL(n+1)$ is the
\textit{standard tractor bundle} $\mc{T}=\G\times_P\rr^{n+1}$. The ray
stabilised by $P$ gives rise to a canonical oriented line subbundle
$\ce(-1)\subset\mc{T}$, whose sections are referred to as projective
$(-1)$--densities.

We consider the holonomy reduction coming from a parallel
non--degenerate metric on $\mc{T}$.  The basic relation between such
metrics and Einstein metrics in the projective class has been observed
in \cite{armstrong-projective1}. The reduction has been studied
further in Section 3.3 of \cite{CGH-projective}, and our main aim here
is to explain the results obtained there from our current perspective,
which is essentially different. This seems important both from the
point of view of comparison and as motivation for subsequent
examples. A bundle metric on $\mc{T}$ can be viewed as a parallel
section of $S^2\Cal T^*$, which as discussed in \ref{partrac} has a
$G$--type. Linear algebra shows that the decomposition of
$S^2{\rr^{n+1}}^*$ into orbits of $\SL(n+1)$ is described by rank and
signature. To deal with the additional distinction by volume in the
non--degenerate case, we will always rescale our metrics by a constant
in such a way that an orthonormal basis has unit volume. Since we
assume our metric to be non--degenerate, $\mc{O}$ consists of all
inner products on $\rr^{n+1}$ which have some fixed signature $(p,q)$
with $p+q=n+1$.  As shown in \ref{partrac}, a parallel tractor metric
is the same as a holonomy reduction of $(\G,\om)$ of type $\mc{O}$.

\begin{thm*}
  Let $(M,[D])$ be a projective structure endowed with a holonomy
  reduction of type $\Cal O$ given by a parallel metric $\mb{h}$ of
  signature $(p,q)$ on the standard tractor bundle $\Cal T$.
  
  (1) The metric $\mb{h}$ determines a normal solution $\si$ of the
  first BGG operator acting on the line bundle $\ce(2)$ of all metrics
  on $\ce(-1)\subset\Cal T$, i.e.
$$ \nabla_{(a}\nabla_{b}\nabla_{c)}\si + 4P_{(ab}\nabla_{c)}\si+
  2\big(\nabla_{(a}P _{bc)}\big)\si=0.
$$ 
  
  (2) The curved orbit decomposition has the form $M=M_+\cup M_0\cup
  M_-$, where $M_\pm\subset M$ are open and $M_0$ coincides with $\Cal
  Z(\si)$ and (if non--empty) consists of embedded hypersurfaces.

  (3) The induced Cartan geometry on $M_+$ (respectively $M_-$) is
  given by a pseudo--Riemannian metric $g_\pm$ of signature $(p-1,q)$
  (respectively $(p,q-1)$) whose Levi--Civita connection lies in the
  projective class.
  
  (4) If $M_0$ is non--empty then it naturally inherits a conformal
  structure of signature $(p-1,q-1)$ via the induced Cartan geometry.
\end{thm*}
\begin{proof}
 The space $\mc{O}$ splits into $P$-orbits as $\Cal O=\Cal O_+\cup\Cal
 O_0\cup\Cal O_-$ according to the restriction of an inner product to
 the distinguished ray $\rr_+ X\in\rr^{n+1}$. On the homogeneous model
 $S^n=G/P$, a parallel section of $S^2\Cal T^*$ is determined by an
 element of $S^2\Bbb R^{(n+1)*}$, so for the given $G$--type, this is
 just an inner product $\langle\ ,\ \rangle$ of signature $(p,q)$ on
 $\Bbb R^{n+1}$. It is easy to see (compare with Section 3.3 of
 \cite{CGH-projective}) that the corresponding normal solution is the
 projective polynomial on $S^n$ induced by the homogeneous polynomial
 $\langle x,x\rangle$ of degree two. In particular, the zero set of
 this polynomial is a smooth embedded hypersurface and coincides with
 $(S^n)_0$.  Via Theorems \ref{2.5} and \ref{3.1} this carries over to the
 curved case, which proves (1) and (2).
  
  Hence we turn to the induced Cartan geometries on the curved orbits.
  According to Theorem \ref{2.5} they have type $(H,H\cap
  P)$, where $H$ is the stabiliser of some element in the orbit in
  question. Let us start with the case $h_+\in\Cal O_+$,
  i.e.~$h_+(X,X)>0$. Then of course $H=G_{h_+}$ is isomorphic to
  $SO(p,q)$ and $H\cap P$ is the stabiliser of a positive ray. But
  elements of $H$ preserve norms, so any element of $H\cap P$ has to
  preserve any vector in the positive ray. Hence $H\cap P$\ is the
  isotropy group $H_X\cong\SO(p-1,q)$ of a unit vector. A Cartan
  geometry of type $(\SO(p,q),\SO(p-1,q))$ is well known to be
  equivalent to a pseudo--Riemannian metric of signature $(p-1,q)$
  together with a metric connection, see Sections 1.1.1 and 1.1.2 of
  \cite{cap-slovak-par}.  Since the canonical Cartan connection
  associated to a projective structure is always torsion free, part
  (\ref{thm-curvedorbits-curv}) of Theorem \ref{2.5} implies that the
  induced Cartan geometries are torsion free. Hence in each case the
  corresponding metric connection in question is torsion free, and
  hence is the Levi-Civita connection. Since the induced Cartan
  geometries are simply obtained by restricting the projective Cartan
  connection, it follows that this Levi-Civita connection lies in the
  projective class. The description of $\Cal O_-$ is completely
  parallel, so the proof of (3) is complete.
  
  (4): Here we again have $H\cong SO(p,q)$, but $H\cap P\subset H$ now
  is the stabiliser of an isotropic ray in the standard
  representation. This is a parabolic subgroup $\underline{P}$ of
  $\SO(p,q)$ and Cartan geometries of this type correspond to
  pseudo--Riemannian conformal structures of signature $(p-1,q-1)$, see
  also Section \ref{Fefferman} below.
\end{proof}

\subsection{Consequences of normality}\label{consequ-norm}
Most of the analysis of a parallel metric on the projective standard
tractor bundle in \ref{sec-exmp-proj} is valid for an arbitrary Cartan
geometry of type $(SL(n+1),P)$. Only in the last part we used that
torsion freeness of the Cartan geometry implies that the induced
Cartan geometries on the open orbits produce Levi-Civita connections
of the induced metric, and not just any metric connection. In the next
step, we will use the fact that we are dealing with the normal Cartan
geometries associated to the underlying projective structure, so an
additional normalisation condition on the Cartan curvature
$K\in\Om^2(\G,\sl(n+1))$, respectively on the corresponding curvature
function $\ka:\G\goesto\La^2(\g/\p)^*\t\g$ is available.

This normalisation condition on the one hand requires $\om$ to be
torsion free, i.e., $\ka$ to have values in $\La^2(\g/\p)^*\t\p$.
Then for $X_1,X_2\in\g/\p$\ and $Y\in\p\subset \sl(n+1)$\ we have that
$[\ka(u)(X_1,X_2),Y]\in\p$, and therefore $\ka(u)(X_1,X_2)$\ factors
to a linear map $\ka_0(u)(X_1,X_2):\g/\p\to\g/\p$.  Via the
identification $\g/\p\cong \rr^n$ we can view $\ka_0(u)$ as an element
of $\La^2{\rr^{n}}^*\t L(\rr^n,\rr^n)$. Now the second part of the
normalisation condition on $\om$ implies that $\ka_0$ is completely
trace-free. This says that $\ka_0(u)\in\La^2{\rr^{n}}^*\t \sl(\rr^n)$
and also the Ricci-type contraction of $\ka_0(u)$ vanishes, i.e.,
\begin{align*}
  \mr{tr}(W\mapsto \ka_0(u)(W,Y)Z)=0
\end{align*}
for all $Y,Z\in\rr^n$. Now we can analyse the consequences for the
induced Cartan geometries. 

\begin{prop*}\label{projective-metric}
  (1) If the open orbit $M_+$ (respectively $M_-$) is non--empty, then
  the induced metric $g_+$ (respectively $g_-$) is Einstein with
  positive (respectively negative) Einstein constant, i.e.~$\Ric(g_+)$
  is a positive multiple of $g_+$ while $\Ric(g_-)$ is a negative
  multiple of $g_-$.
  
  (2) If the closed curved orbit $M_0$ is non--empty, then the induced
  Cartan geometry of type $(\SO(p,q),\underline{P})$ is normal.
\end{prop*}
\begin{proof}
  (1) We consider the case of $M_+$ and indicate the necessary changes
  for $M_-$ in the end. Throughout the proof, we work in a point
  $u\in\G$ which is contained in the reduced Cartan subbundle. Since
  we are dealing with an open orbit we get $\h/(\h\cap\p)\cong\g/\p$,
  so part (iii) of Theorem \ref{2.5} shows that the value $\ka(u)$ of
  the curvature function $\ka$ of $\om$ coincides with the value of
  the curvature function of the induced Cartan connection. This also
  shows that torsion freeness of $\om$ implies that $\ka(u)\in
  \La^2{\rr^{n}}^*\t\h$.
  
  Since $\Cal T$ is associated to a principal $SL(n+1,\Bbb
  R)$--bundle, it has a distinguished volume form, and rescaling the
  tractor metric by a constant, we may assume that orthonormal bases
  have unit volume. We can thus work in matrix representations with
  respect to orthonormal bases. Then $\frak{so}(p,q)$ has the form
$$
\left\{\begin{pmatrix}0 & -Y^t\Bbb I_{p-1,q} \\ Y & A\end{pmatrix}:
  Y\in\rr^n, A\in \frak{so}(p-1,q)\right\},
$$
where $\Bbb I_{p-1,q}$ is diagonal with $p-1$ entries equal to 1
and $q$ entries equal to $-1$. Since the subspaces $\rr^n$ and
$\frak{so}(p-1,q)$ in $\h$ are invariant under $SO(p-1,q)$, the
components of the Cartan connection in the two subspaces are
individually equivariant. The $\rr^n$--component $\th$ defines a
soldering form on the bundle which is used to carry over the inner
product $\langle\ ,\ \rangle$ on $\Bbb R^n$ corresponding to $\Bbb
I_{p-1,q}$ to the metric $g_+$ on the tangent spaces of $M_+$. The
$\frak{so}(p-1,q)$--component $\ga$ is a principal connection, which
induces a metric connection on the tangent bundle, and by torsion
freeness, this is the Levi--Civita connection of $g_+$.

From the definition of the curvature of a Cartan connection it follows
that the curvature $K_+$ of the induced Cartan connection is given by
\begin{equation}\label{KRdiff}
    K_+(u)(\xi,\eta)={R}(u)(\xi,\eta)+[\th(\xi),\th(\eta)],
\end{equation}
where $R$ is the curvature of $\ga$, and hence the Riemann curvature,
and the last bracket is in $\frak{so}(p,q)$. Now one immediately
computes that for $Y_1,Y_2,Z\in\rr^n\subset\frak h$, we get
\begin{equation}
  \label{bracket}
  [[Y_1,Y_2],Z]=\langle Y_1,Z\rangle Y_2-\langle Y_2,Z\rangle Y_1
\end{equation}
Using this, one easily calculates that the Ricci type contraction
of $Y_1,Y_2\mapsto [Y_1,Y_2]$ is given by $-(n-1)$ times the inner
product $\langle\ ,\ \rangle$. Since the left hand side of
\eqref{KRdiff} has vanishing Ricci type contraction by normality, we
conclude that the Ricci type contraction $\Ric(g_+)$ of $R$ equals
$(n-1)g_+$, so $g_+$ is positive Einstein.  

In the case of $M_-$, the first basis vector used to define the matrix
representation must be chosen to be negative. But then the entries of
a matrix in $\rr^n\subset\frak{so}(p,q)$ must be $Y$ and $Y^t\Bbb
I_{p,q-1}$. This causes a sign change in formula \eqref{bracket} and
hence in the Ricci--type contraction, so on obtains
$\Ric(g_-)=-(n-1)g_-$.

\medskip

(2) Again we work in a point $u\in\G$ which lies in the reduced Cartan
bundle over the curved orbit, which means that we work in a basis for
$\Cal T_x$, which is adapted to the tractor metric. We choose this
basis in such a way that the first basis vector $X$ spans the
distinguished line (which is isotropic for the tractor metric in this
point), the last basis vector is isotropic and pairs to one with $X$
under the tractor metric. Then we choose an orthonormal basis for the
orthocomplement of the plane spanned by these two vectors to complete
our basis. The normalisation condition on the projective Cartan
curvature implies torsion freeness and that its $\frak g_0$--component
has values in $\frak{sl}(n)$. Moreover, it has to be skew symmetric
with respect to the tractor metric, so altogether it must be of the
form
\begin{equation}
  \label{m0curv}
\ka(\xi,\eta)=\begin{pmatrix} 0 & Z(\xi,\eta) & 0 \\ 0 & A(\xi,\eta) &
  -\Bbb I(Z(\xi,\eta))^t \\ 0 & 0 & 0\end{pmatrix}.  
\end{equation}
Here the blocks are of size $1$, $n-1$, and $1$ and $\Bbb I=\Bbb
I_{p-1,q-1}$. Finally the normalisation condition tells us, that the
Ricci type contraction over the lower right $n\x n$--block has to
vanish. This coincides with the Ricci type contraction of $A$ taken
over $X^\perp/\Bbb RX$.

Now we know that restricted to the tangent space of the reduced Cartan
bundle, the projective Cartan connection restricts to the reduced
Cartan connection. In particular, its soldering form must have values
in $X^\perp/\Bbb RX$, so this corresponds to
$\h/(\h\cap\p)\subset\g/\p$ and thus represents the tangent spaces to
$M_0$. Moreover, by part (iii) of Theorem \ref{2.5}, \eqref{m0curv}
coincides with the curvature of the reduced Cartan connection. Now the
normalisation condition on a conformal Cartan connection is torsion
freeness plus vanishing of the Ricci--type contraction of the $\frak
g_0$--component of the curvature function. Since the latter component
is represented by $A(\xi,\eta)$, normality of the induced Cartan
geometry follows.
\end{proof}

\subsection{Hermitian metrics on complex projective standard tractors}
\label{complex-projective}
There is an almost complex version of metrics on the projective
standard tractor bundle. The geometry behind this is that of an almost
complex projective structure, which is a generalisation of the notion
of $h$--projective structure, see  
e.g.\cite{matveev,Apostolov-Calderbank-Gauduchon} and references therein. This
case is significantly more complicated than the real version in
several respects. Therefore, we will only derive some basic facts
here, and study it in more detail elsewhere.

Almost complex classical projective structures can be equivalently
described as parabolic geometries of type $(G,P)$, where
$G=SL(n+1,\Bbb C)$ and $P\subset G$ is the stabiliser of a complex line in the
standard representation $\Bbb C^{n+1}$ of $G$. However, $G$ and $P$
are viewed as real Lie groups and likewise one has to consider their
Lie algebras as real Lie algebras. Doing this, one obtains a geometry
which is much more general than just the obvious holomorphic analog of
a classical projective structure. As far as we know, the parabolic
geometry approach to this structure is not developed in detail in the
literature, a brief account can be found in Section 4.6 of
\cite{cap-correspondence}.

Explicitly, one has to consider manifolds $M$ of real dimension $2n$
endowed with an almost complex structure $J:TM\to TM$. Then there is a
complex version of the projective equivalence of linear connections on
$TM$. Consider a $(1,0)$--form $\Ups$, i.e.~$\Ups(x)$ is a complex
linear map $T_xM\to\Bbb C$ (with respect to $J_x$) for each $x$. Then,
defining the action of complex numbers on tangent vectors via $J$, one
defines complex projective equivalence by 
$$
\tilde D_\xi\eta=D_\xi\eta+\Ups(\xi)\eta+\Ups(\eta)\xi, 
$$ for all vector fields $\xi$, $\eta$. Note that this does not imply
projective equivalence in the real sense, since complex linear
combinations of $\xi$ and $\eta$ occur on the right hand side. The use
of complex linear combinations leads to the fact that $DJ=0$ implies
that $\tilde DJ=0$ for any equivalent connection. 

Connections which are equivalent in this sense have the same
torsion. However, a connection preserving an almost complex structure
cannot be assumed to be torsion free, since it is well known that for
such a connection, the $(0,2)$--component of the torsion is (up to
a nonzero factor) given by the Nijenhuis tensor of $J$. On the other
hand, given an almost complex structure one can always find a
connection $D$ compatible with $J$ for which the torsion is of type
$(0,2)$.

Hence an almost complex projective structure is defined to  be an
almost complex manifold $(M,J)$ together with an equivalence class
$[D]$ of connections such that $DJ=0$ and the torsion of $D$ is of
type $(0,2)$. 

\medskip

Now the holonomy reduction we want to consider in this case is related
to the homogeneous space $\Cal O$ of $G=SL(n+1,\Bbb C)$ which consists
of all Hermitian inner products on $\Bbb C^{n+1}$ which are
non--degenerate with some fixed signature $(p+1,q+1)$, where
$p+q=n-1$. Starting from an almost complex projective structure, one
forms the corresponding normal parabolic geometry of type $(G,P)$.
Forming the associated bundle to the Cartan bundle with respect to
(the restriction to $P$ of) the standard representation $\Bbb C^{n+1}$
of $G$, one obtains the standard tractor bundle $\Cal T$ of the almost
complex projective structure. By construction, this is a complex
vector bundle of complex rank $n+1$ and the complex line, in the
standard representation, stabilised by $P$ gives rise to a complex
line subbundle $\Cal E\subset\Cal T$. A holonomy reduction of type
$\Cal O$ is equivalent to a Hermitian bundle metric $\mb{h}$ on $\Cal
T$ which is non--degenerate of signature $(p+1,q+1)$ and parallel for
the canonical connection.

\begin{thm*}
  Let $(M,J,[D])$ be an almost complex projective structure and
  suppose that we have fixed a holonomy reduction of type $\Cal O$, as given
   by a parallel Hermitian metric $\mb{h}$ of signature
  $(p+1,q+1)$ on the standard tractor bundle $\Cal T$. Then we have:
  
  (1) The metric $\mb{h}$ determines a normal solution $\si$ of the
  first BGG operator acting on sections of the real line bundle $\Cal
  H_0$ of Hermitian metrics on the (complex) density bundle $\Cal
  E\subset\Cal T$. This BGG operator is of second order and has values
  in the space of symmetric, anti--Hermitian bilinear maps $TM\x
  TM\to\Cal H_0$. 
  
  (2) The curved orbit decomposition has the form $M=M_-\cup M_0\cup
  M_+$, where the first and last curved orbits are open and given by
  those points where the metric defined by $\si$ is negative definite,
  respectively positive definite. If non--empty, the curved orbit
  $M_0$ is an embedded hypersurface which coincides with $\Cal
  Z(\si)$.
  
  (3) On $M_\pm$ one obtains induced Hermitian metrics of signature
  $(p,q+1)$, respectively $(p+1,q)$, together with metric connections.
  If the initial structure is torsion free, then the metric
  connections are the Levi--Civita connections, so one actually
  obtains K\"ahler structures on $M_{\pm}$.

  (4) If $M_0\neq\emptyset$, then it inherits a Cartan geometry of
  type $(SU(p+1,q+1),\underline{P})$, where $\underline{P}$ is the
  stabiliser of an isotropic line. If the initial structure is torsion
  free, then this induces an (integrable) CR structure of signature
  $(p,q)$ on $M_0$.
\end{thm*}
\begin{proof}
The tractor bundle $V$ in question is induced by the space of all
Hermitian bilinear forms on $\Bbb C^{n+1}$. The subspace of those
forms, which vanish on the complex line stabilised by $P$ is evidently
$P$--invariant and has codimension one. Thus it must be the maximal
$P$--invariant subspace and the quotient by this subspace can clearly
be identified with the space of Hermitian bilinear maps on the
distinguished line, which gives the description of $\Cal H_0$ in
(1). The cohomology space inducing $\Cal H_1$, i.e. the target space
for the given first BGG operator, can be either computed directly or
using representation theory methods, and having this part (1) follows
from the general theory discussed in \ref{3.1}.
  
  The stabiliser of any Hermitian inner product from $\Cal O$ is a
  subgroup of $G$ conjugate to $SU(p+1,q+1)\subset SL(n+1,\Bbb C)$.
  The homogeneous model $G/P$ of the geometry is simply the complex
  projective space $\Bbb CP^n$, so according to \ref{P-types} we can
  determine $P$--types by looking at $SU(p+1,q+1)$--orbits on the
  space of complex lines in $\Bbb C^{n+1}$. These orbits are
  determined by the signature of the restriction of the Hermitian
  inner product to a line, so this looks exactly as in the real case,
  and we have $\Cal O=\Cal O_+\cup\Cal O_0\cup\Cal O_-$. This also
  gives us the basic form of the curved orbit decomposition in (2).
  Identifying the stabiliser $H$ of the inner product with
  $SU(p+1,q+1)$ the subgroups $H\cap P$ in the three cases are
  conjugate to $S(U(1)\x U(p,q+1))\cong U(p,q+1)$, a parabolic
  subgroup $\underline{P}\subset SU(p+1,q+1)$, and $S(U(p+1,q)\x
  U(1))\cong U(p+1,q)$, respectively.

  For the $P$--type defined by $\Cal O_+$, the standard way to present
  $Q:=S(U(1)\x U(p,q+1))\subset SU(p+1,q+1)=:H$ is as matrices of the
  form $\begin{pmatrix} \det(A)^{-1} & 0 \\ 0 & A\end{pmatrix}$ with
  $A\in U(p,q+1)$. On the level of Lie algebras, we have
$$
\frak h=\left\{\begin{pmatrix} -\tr(B) & -Z^*\Bbb I_{p,q+1} \\ Z &
    B\end{pmatrix}: Z\in\Bbb C^n, B\in\frak u(p,q+1)\right\},
$$ where $\Bbb I_{p,q+1}$ is the diagonal matrix of size $p+q+1$ with
    first $p$ entries equal to $1$ and last $q+1$ entries equal to
    $-1$.  The Lie algebra $\frak q$ of $Q\subset H$ corresponds to
    the block diagonal part. Hence $\frak h/\frak q$ can be identified
    with $\Bbb C^n$ with the representation of $Q$ on this space given
    by $A\cdot Z=\det(A)AZ$. In particular, the obvious complex
    structure on $\frak h/\frak q$ as well as the standard inner
    product of signature $(p,q+1)$ on this space are invariant under
    the action of $Q$.  Consequently, a Cartan geometry of type
    $(H,Q)$ on a smooth manifold $M$ gives rise to an almost complex
    structure $J$ and a Hermitian (with respect to $J$) metric $\hhh$
    of signature $(p,q+1)$.  Finally, the Cartan geometry also gives
    rise to a principal connection, which can be equivalently encoded
    as a linear connection $\nabla$ on the tangent bundle which is
    compatible both with $J$ and with $\hhh$.  Together, $J$, $\hhh$,
    and $\nabla$ completely determine the Cartan geometry. For the
    orbit $\Cal O_-$, the situation is completely parallel, with the
    only difference that $\hhh$ has signature $(p+1,q)$ rather than
    $(p,q+1)$. This completes the proof of the first part of (3).

For $\Cal O_0$, we get the stabiliser of an isotropic line as the
subgroup in $H$, and it is well known that Cartan geometries of the
corresponding type are related to partially integrable almost CR
structures of signature $(p,q)$, see also \ref{CR-almost-Einstein}
below. This also implies the first part of (4).

On the homogeneous model $\Bbb CP^n$, a parallel metric on the
standard tractor bundle corresponds to a fixed Hermitian inner product
$\underline{\mb{h}}$ of signature $(p+1,q+1)$ on $\Bbb C^{n+1}$.  The
orbit decomposition is just given by the signature of the restriction
of $\underline{\mb{h}}$ to the complex line determined by a point in
$\Bbb CP^n$ as described above. It is well known that the spaces of
positive, respectively negative, lines are open and they are the complex
hyperbolic spaces of signature $(p,q+1)$ and of signature$(p+1,q)$,
respectively. The space of isotropic lines is a quadric and, in
particular, a smooth embedded hypersurface. It is the homogeneous
model of (partially integrable almost) CR structures of hypersurface
type, which are non--degenerate of signature $(p,q)$. Since $\Cal O_0$
is exactly the zero set of the normal solution $\underline{\si}$
determined by $\underline{\mb{h}}$, we obtain (2).

To complete the proof, let us assume that the initial Cartan geometry
$(p:\Cal G\to M,\om)$ is torsion free (which is equivalent to the fact
that we deal with a holomorphic projective structure, see
\cite{cap-correspondence}). Then by part (iii) of Theorem \ref{2.5} also the
induced Cartan geometries are torsion free. It is well known that
torsion free Cartan geometries of type $(SU(p+1,q+1),\underline{P})$
are equivalent to (integrable) non--degenerate CR structures of
signature $(p,q)$, so the proof of (4) is complete.

For the curved orbits $M_\pm$ torsion freeness of the induced Cartan
geometries implies that the induced metric connections are torsion
free as in the real case. Since these connections preserve the
Hermitian metrics, these metrics actually must be K\"ahler.
\end{proof}

Analysing the consequences of normality is similar to the real case,
but significantly more complicated, in particular if one drops the
assumption of torsion freeness. Thus we just indicate some basic facts
in the torsion--free case here, and pick up the detailed discussion
elsewhere. First one has to analyse the relation between the Cartan
curvatures of the initial Cartan connection of type $(G,P)$ and the
induced Cartan connections.  For the open orbits, this is similar to
the discussion in the proof of Proposition \ref{projective-metric},
and one verifies that K\"ahler--Einstein metrics are induced on
$M_\pm$. Moreover, parallel to \cite[Theorem 3.3]{CGH-projective} one
shows that these metrics are complete if one starts from a complex
projective structure on a compact manifold. Hence in the torsion free
case, one obtains a compactification of a complete K\"ahler--Einstein
manifold by adding a CR structure at infinity.

If one does not assume the original structure to be torsion free, the
induced connection on $M_\pm$ will differ (in a controlled way) from
the Levi-Civita connection, and also the normalisation condition for
the Cartan connection of type $(G,P)$ becomes significantly more
involved. To describe the induced geometry on $M_0$, one first has to
check when the induced Cartan connection of type
$(SU(p+1,q+1),\underline{P})$ is regular, since then it induces a
partially integrable almost CR structure on $M_0$.

\subsection{Fefferman-type constructions}\label{Fefferman}
We next outline examples of holonomy reductions for conformal
structures. Let $M$ be a smooth manifold of dimension $n\geq 3$. Then
a conformal structure of signature $(p,q)$ on $M$ is given by an
equivalence class $[g]$ of pseudo--Riemannian metrics of signature
$(p,q)$ on $M$. Here two metrics $g$ and $\hat g$ are considered
equivalent if there is a positive smooth function $f:M\to\Bbb R$ such
that $\hat g=fg$. It is a classical result going back to E.~Cartan
that an oriented conformal structure can be equivalently described as
a parabolic geometry of type $(G,P)$, where $G=SO(p+1,q+1)$ and
$P\subset G$ is the stabiliser of an isotropic ray in the standard
representation $\Bbb R^{p+1,q+1}$ of $G$. The corresponding grading on
the Lie algebra $\frak g$ of $G$ has the form $\frak g=\frak
g_{-1}\oplus\frak g_0\oplus\frak g_1$, with $\frak g_0=\frak{co}(p,q)$
(the Lie algebra of the conformal group of signature $(p,q)$), $\frak
g_{-1}\cong\Bbb R^{p,q}$ and $\frak g_1\cong (\Bbb R^{p,q})^*$ as
representations of $\frak g_0$.

We will need the \textit{standard tractor bundle}, which is associated
to the (restriction to $P$ of the) standard representation $\Bbb
R^{p+1,q+1}$ of $G$. It inherits a canonical linear connection
from the conformal Cartan connection. By construction, it carries a
natural bundle metric $h$ of signature $(p+1,q+1)$, the tractor
metric, and the isotropic line stabilised by $P$ gives rise to a line
subbundle $\Cal T^1\subset\Cal T$, whose fibres are isotropic with
respect to $h$. The orthogonal spaces to these preferred isotropic
lines fit together to form a subbundle $\Cal T^0\subset\Cal T$ of
corank one, so we obtain a filtration $\Cal T^1\subset \Cal
T^0\subset\Cal T$ of the tractor bundle.

The first group of examples is related to generalised Fefferman
constructions. In these cases, the curved orbit decomposition is
trivial, since there is only one $P$--type. This does not mean that
our results are vacuous in these cases, however. On the one hand, we
may conclude from just looking at the homogeneous model that also in
curved cases the solutions obtained from such a restriction cannot
have any zeroes. On the other, we immediately see that in this case we
get a reduced Cartan geometry over the whole manifold. Since these
cases are quite well studied in the literature, we only discuss them
very briefly. 

The most classical example of this situation comes up when both $p$
and $q$ are odd, say $p=2p'+1$ and $q=2q'+1$. Then we have the
subgroup $U(p'+1,q'+1)\subset SO(2p'+2,2q'+2)$, which looks like a
good candidate for a holonomy group. In the language of Section 2, we
have to look at the homogeneous space $\Cal O$ of $G$ which consists
of all complex structures $J$ on $\Bbb R^{(2p'+2,2q'+2)}$, which are
orthogonal (or equivalently skew--symmetric) with respect to the
standard inner product. It is clear then that such a holonomy reduction
in the curved case is given by the choice of a complex structure $\Cal
J$ on the bundle $\Cal T$, which is skew symmetric with respect to the
tractor metric $h$ and parallel with respect to (the connection
induced by) the standard tractor connection.

The best way to view $\Cal J$ is as a parallel section of the
\textit{adjoint tractor bundle} $\frak{so}(\Cal TM)=:\Cal AM$. This
bundle is induced by (the restriction to $P$ of) the adjoint
representation $\frak{so}(2p'+2,2q'+2)$ of $G$. The natural quotient
$\Cal H_0$ in this case is the tangent bundle $TM$, and the
corresponding first BGG operator is the conformal Killing operator,
whose kernel consists of all infinitesimal conformal isometries. It
turns out (see e.g.\ \cite{cap-infinitaut}) that normal solutions are
conformal Killing fields which insert trivially into the Weyl curvature
and into the Cotton--York tensor.

The homogeneous model $G/P$ of conformal geometry is the space of
isotropic rays in $\Bbb R^{2p'+2,2q'+2}$, which is diffeomorphic to
$S^{2p'+1}\x S^{2q'+1}$. Since the group $U(p'+1,q'+1)$ evidently acts
transitively on the space of real isotropic rays in $\Bbb
C^{p'+1,q'+1}$, we conclude from (\ref{Modeltypemap}) that indeed
there is only one $P$--type in this case. In particular, this implies
that the conformal Killing vector field underlying $\Cal J$ is nowhere
vanishing.  Hence it determines a one--dimensional foliation of $M$ from
which we can form local leaf spaces.

It then turns out that the holonomy reduces further to $SU(p'+1,q'+1)$
\cite{Leitcomm,cap-gover-cr}, and that the restricted Cartan geometry
over all of $M$ can be viewed as a Cartan geometry over a local leaf
space $N$ for the foliation mentioned above. This is of type
$(SU(p'+1,q'+1),P_{SU})$, where $P_{SU}$ is the stabiliser of an
isotropic complex line. This geometry turns out to be regular and
normal (see \cite{cap-gover-cr}), thus giving rise to CR structure of
signature $(p',q')$ on $N$. The space $M$ thereby becomes locally
conformally isometric to the Fefferman space of this CR structure. The
fact that the canonical Cartan connection associated to a CR structure
agrees with the one associated to its Fefferman space has rather
strong consequences, see \cite{cap-gover-cr-tractors}.

\medskip

There are variants of this situation in which one obtains similar
results. For example, there is an analog based of the inclusion of the
split real form of the exceptional Lie group $G_2$ into
$SO(3,4)$. This is related to generic rank two distributions on
manifolds of dimension five studied in Cartan's celebrated
``five--variables paper'' \cite{cartan-cinq}, and the canonical
conformal structure induced by such a distribution obtained in
\cite{nurowski-metric}. On the level of tractor bundles, this holonomy
reduction is related to a parallel tractor three form (i.e. a section
of $\La^3\Cal T^*$) of certain algebraic type. The underlying
geometric object is a normal conformal Killing two--form, which by our
result is nowhere vanishing. In this case no leaf spaces are involved 
 and the restricted Cartan
connection is again normal. This leads to a characterisation of such
holonomy reductions and it has further strong consequences, see
\cite{mrh-sag-rank2}.

A very similar construction applies to split signature conformal
structures in dimension $6$, where the relevant subgroup is
$Spin(8)\subset SO(4,4)$, and one has to consider the fourth exterior
power of the standard tractor bundle rather than the third one. A
holonomy reduction then gives rise to a generic distribution of rank
three, to which the initial conformal structure is canonically
associated as first shown in R.~Bryant's thesis, see
\cite{bryant-3planes,mrh-sag-twistors} for recent accounts.

\subsection{Almost Einstein scales}\label{almost-Einstein}
We next study the simplest example of a holonomy reduction for
conformal structures, namely the existence of a parallel section $s$
of the standard tractor bundle $\Cal T$ of a conformal structure
$(M,[g])$. This was actually the first case in which the zeroes of a
solution of a first BGG operator were studied using parabolic geometry
methods, see \cite{gover-aes}  and this motivated
 many of the developments that led to
this article. We treat this case here to illustrate how it fits into
the more general machinery (the latter also simplifying many aspects).

Recall from Section \ref{Fefferman} above that for a conformal
structure of signature $(p,q)$ the standard tractor bundle has rank
$p+q+2$ and is endowed with a canonical parallel metric $h$ of
signature $(p+1,q+1)$, as well as a line subbundle $\Cal
T^1\subset\Cal T$ whose fibres are isotropic with respect to
$h$. Following a standard convention in conformal geometry, we shall
write $\Cal E[1]:=(\Cal T^1)^*$.  Note that for a parallel section $s$
of $\Cal T$, the function $h(s,s)$ must be constant, and up to a
constant rescaling of $s$, the possible $G$ types in this case are
distinguished by the fact that $h(s,s)$ is positive, zero, or
negative, respectively.

\begin{thm*}
  Suppose that $(M,[g])$ is an oriented conformal pseudo--Rie\-mann\-ian
  structure of signature $(p,q)$ and that $s$ is a parallel section of
  the standard tractor bundle $\Cal T$ of $M$. Then $s$ projects onto
  a normal solution $\si\in\Ga(\Cal E[1])$ of a first BGG operator.
  
  (1) Suppose that $h(s,s)>0$ (respectively $h(s,s)<0$). Then the
  curved orbit decomposition has the form $M=M_+\cup M_0\cup M_-$,
  where $M_+$ and $M_-$ are open and $M_0$ coincides with the zero
  locus of $\si$ and (if non--empty) consists of embedded
  hypersurfaces.  Moreover, on $M_\pm$, there is an Einstein metric
  in the conformal class, whose Einstein constant is negative if
  $h(s,s)>0$ and positive if $h(s,s)<0$. The curved orbit $M_0$ is
  empty if $p=0$ (respectively $q=0$), otherwise it inherits a
  conformal structure of signature $(p-1,q)$ (respectively $(p,q-1)$).
  The induced Cartan geometry on $M_0$ is the normal Cartan geometry
  determined by this conformal structure.
  
  (2) Suppose that $h(s,s)=0$. Then the curved orbit decomposition has
  the form $M=M_1^+\cup M_1^-\cup M_2\cup M_3^+\cup M_3^-$, where
  $M_1^\pm$ are open, $M_2\cup M_3^\pm=\Cal Z(\si)$, 
$M_2$ (if non--empty)
  consists of smoothly embedded hypersurfaces and $M_3^\pm$ (if non--empty)
  consists of isolated points.  If $p=0$ or $q=0$, then $M_2$ must be
  empty, so $\si$ can only have isolated zeros in this case.
  Otherwise, if $M_3^+$ or $M_3^-$ is non--empty, then also $M_2$ has
  to be non--empty.
  
  On $M_1^\pm$ there is a Ricci--flat metric in the conformal class.
  If non--empty, the curved orbit $M_2$ locally fibres over a smooth
  manifold $N$ with one--dimensional fibres. If vectors tangent to
  these fibres insert trivially into the Weyl tensor and the
  Cotton--York tensor of the initial conformal structure, then the leaf
  space $N$ inherits a canonical conformal structure of signature
  $(p-1,q-1)$.
\end{thm*}
\begin{proof}
  (1) Let us assume that $h(s,s)>0$. Then the stabiliser $H$ of $s$ is
  isomorphic to $SO(p,q+1)\subset SO(p+1,q+1)$. There are three
  possible $P$--types in this case, defined by the fact that the inner
  product of $s(x)\in\Cal T_x$ with a generator of the distinguished
  isotropic ray is positive, zero, or negative. (Note that $s(x)$
  cannot lie in the distinguished isotropic line.) The irreducible
  quotient of $\Cal T$ for conformal structures is a density bundle
  usually denoted by $\Cal E[1]$ which is realised as $\Cal T/(\Cal
  T^1)^\perp$. Hence we conclude that the curved orbit defined by
  $s(x)\perp \Cal T^1_x$ is exactly the zero set of the induced normal
  solution $\si=\Pi(s)$. On the homogeneous model, $s$ is determined
  by a fixed positive vector in $\Bbb R^{p+1,q+1}$ so the orbit in
  question is the subspace of isotropic lines contained in a
  hyperplane, and thus a smoothly embedded hypersurface in $G/P$. This gives
  the description of curved orbits in (1). If $p=0$, then the
  restriction of $h$ to $s(x)^\perp$ is negative definite, and hence
  this subspace does not contain any isotropic lines, so
  $M_0=\emptyset$.
  
  To describe the induced Cartan geometry on $M_\pm$, we have to
  understand the stabiliser of the distinguished isotropic ray in
  $H\cong SO(p,q+1)$, and we know that this ray is transversal to the
  hyperplane $s(x)^\perp$ stabilised by $H$ and different from the
  line $\Bbb R\cdot s(x)$. If we project $\Cal T^1_x$ orthogonally
  into $s(x)^\perp$ we thus obtain a line. This line has to be
  negative, since together with the positive line $\Bbb R\cdot s(x)$
  it spans a plane which contains the isotropic line $\Cal T^1_x$. A
  moment of thought shows that the stabiliser of this negative line in
  $H$ coincides with the stabiliser of the isotropic ray in that
  group, so $H\cap P=SO(p,q)\subset SO(p,q+1)$.
  
  Thus we are in the same situation as in the example of the parallel
  metric on the projective standard tractor bundle. The induced Cartan
  geometry on the open curved orbit is equivalent to a
  pseudo--Riemannian metric of signature $(p,q)$ together with a
  metric connection on the tangent bundle. Using torsion freeness of
  the initial conformal Cartan connection, one concludes that the
  metric connection must be the Levi--Civita connection. From the
  normality of the conformal Cartan connection one then deduces that
  the induced metric is Einstein, with the sign of the Einstein
  constant determined by $h(s,s)$.
  
  The type of the induced Cartan geometry on $M_0$ is even easier to
  determine, since here $H\cap P$ simply is the stabiliser of an
  isotropic ray in the standard representation of $H=SO(p,q+1)$.
  Hence the induced Cartan geometry on the closed curved orbit
  determines an oriented conformal structure of signature $(p-1,q)$.
  It is straightforward to prove that normality of the initial
  conformal Cartan connection implies that also the induced Cartan
  connection over the closed curved orbit is normal. Since the
  discussion for $h(s,s)<0$ is completely parallel, this completes the
  proof of (1).

\medskip

(2) If $h(s,s)=0$, then $s$ spans an isotropic line subbundle in the
standard tractor bundle. In this case, the stabiliser $H$ is the
stabiliser of an isotropic vector in $\Bbb R^{p+1,q+1}$ and thus
isomorphic to $SO(p,q)\rtimes\Bbb R^{p,q}$ (so it is isomorphic to a
codimension one subgroup in the parabolic subgroup $P$). It is also
clear that there are five possible $P$--types, according to the cases
that $s(x)$ lies in the preferred isotropic ray ($M_3^+$), lies in its
negative ($M_3^-$), lies not in the preferred line but in its
orthocomplement ($M_2$), or has positive respectively negative inner
product with the preferred ray ($M_1^\pm$). Notice however, that if
either $p=0$ or $q=0$, the initial vector space $\Bbb R^{p+1,q+1}$ is
Lorentzian and hence does not contain two perpendicular isotropic
lines. Thus in this special case only the first two and the last two
$P$--types can occur.

In the homogeneous model $G/P$, our parallel section is determined by
an isotropic vector in $\Bbb R^{p+1,q+1}$, and it is evident that the
first two and the last two of the five $P$--orbits consist of isolated
points, and open subsets, respectively. For the middle $P$--type, we
observe that taking the tractor inner product of $s$ with the
preferred ray defines the solution $\si\in\Ga(\Cal E[1])$ underlying
$s$. The orbit under consideration consists of the zero set of this
section except for the two points where $s$ lies in $\Cal T^1$.  But
it is evident that these two points are the only ones in which the
hypersurface orthogonal to $s$ is not transversal to the tangent space
of the null--cone, which implies that our orbit is a smooth
hypersurface in the homogeneous model. Hence we obtain the claimed
form of the curved orbit decomposition.

To discuss the induced Cartan geometry on $M_1^\pm$, we have seen that
$H\cong SO(p,q)\rtimes \Bbb R^{p,q}$, and $H\cap P$ is the stabiliser
of an isotropic line in there, which is not perpendicular to the
vector stabilised by $H$. A moment of thought shows that the
isomorphism $H\cong SO(p,q)\rtimes \Bbb R^{p,q}$ can be chosen in such
a way that $H\cap P=SO(p,q)$ viewed as a subgroup in the obvious way.
Hence it is again clear that a Cartan geometry of type $(H,H\cap P)$
is a pseudo--Riemannian metric of signature $(p,q)$ together with a
metric connection. In contrast to the situation in the proof of
Theorem \ref{sec-exmp-proj} the Cartan curvature here simply agrees
with torsion and curvature of that connection. Torsion freeness then
implies that the connection is the Levi--Civita connection of the
induced metric, and normality of the initial conformal Cartan
connection shows that we actually get a Ricci flat metric in the
conformal class.

Hence it remains to discuss the induced Cartan geometry on $M_2$. Here
it is easiest to describe the Lie algebras $\frak h$ and $\frak
h\cap\frak p$ in a basis which starts with the distinguished isotropic
vector $v$, next a perpendicular isotropic vector spanning the line
stabilised by $P$ and then completing this appropriately to a basis.
Then we get
$$
\left\{
\begin{pmatrix}
  0 & 0 & Z & z' & 0 \\ 0 & a & W & 0 & -z' \\ 0 & 0 & A & -\Bbb I W^t
  & -\Bbb I Z^t \\ 0 & 0 & 0 & -a & 0 \\ 0 & 0 & 0 & 0 & 0
\end{pmatrix}\right\}\subset 
\left\{
\begin{pmatrix}
  0 & z & Z & z' & 0 \\ 0 & a & W & 0 & -z' \\ 0 & X & A & -\Bbb I W^t &
  -\Bbb I Z^t \\ 0 & 0 & -X^t\Bbb I & -a & -z \\ 0 & 0 & 0 & 0 & 0
\end{pmatrix}\right\}. 
$$ Here $X,Z^t,W^t\in\Bbb R^{p-1,q-1}$, $a,z,z'\in\Bbb R$, and
$A\in\frak{so}(p-1,q-1)$ and $\Bbb I=\Bbb I_{p-1,q-1}$. Hence $\frak
h/(\frak h\cap\frak p)$, which models the tangent space, can be
identified with $\Bbb R\oplus\Bbb R^{p-1,q-1}$, with the summands
spanned by $z$ and $X$, respectively. The line spanned by $z$ is
invariant under $\frak h$ and will thus give rise to a natural line
subbundle.  
The resulting distribution is of course involutive, so
locally around each point of $M_2$ one can form a local leaf space
$N$. There is a natural subgroup $Q\subset H$ containing $H\cap P$,
whose Lie algebra $\frak q$ is spanned by $\frak h\cap\frak p$ and
$z$. According to Proposition 2.6 of \cite{cap-correspondence} the 
principal $H\cap
P$--bundle over $M_2$ can be locally viewed as a principal $Q$--bundle
over $N$. Moreover, by Theorem 2.7 of that reference, the Cartan
connection on the $H\cap P$--bundle defines a Cartan connection on the
$Q$--bundle if elements of the line subbundle insert trivially into
the Cartan curvature (or equivalently into the Weyl curvature and the
Cotton-York tensor of $[g]$). But from the presentation of the Lie algebras
above, we see that $\frak h/\frak q$, as a module over $Q$ is
equivalent to the standard representation of $CSO(p,q)$ (which
naturally is a quotient of $Q$). Thus, in this case, there is a
natural conformal structure on $N$.
\end{proof}

\begin{remark*}
  (1) The fact that the orbit $M_2$ fibres as claimed in the theorem
  is nicely visible on the homogeneous model. Here the orbit consists
  of isotropic lines contained in $v^\perp$ and different from the
  line spanned by $v$. Now the quotient $v^\perp/\Bbb Rv$ inherits a
  natural inner product of signature $(p,q)$ and we can project our
  isotropic lines to isotropic lines in this quotient with
  one--dimensional fibres.
  
  (2) It is worth noticing that the natural metrics over the open
  orbits which show up in the theorem are determined by the normal
  solution $\si\in\Ga(\Cal E[1])$ underlying the parallel tractor $s$
  in a rather simple way. Recall that a conformal class of metrics can
  be viewed as a canonical section $\mathbf{g}$ of the bundle
  $S^2T^*M\otimes\Cal E[2]$. Hence outside the zero set of $\si$, one
  obtains a metric in the conformal class as
  $\tfrac{1}{\si^2}\mathbf{g}$ and these are the Einstein metrics over
  the open curved orbits.
  
  (3) If $h(s,s)<0$, then it is shown in \cite{gover-PE,gover-aes} that
  locally around $M_0$, one actually obtains a Poincar\'e--Einstein
  metric, and any Poincar\'e--Einstein metric arises in this
  way (of course the metric singularity set may be the boundary of the
  structure).

  (4) We want to point out here that both a parallel metric on a
  projective standard tractor bundle and a parallel section of a
  conformal standard tractor bundle give rise to Einstein metrics on
  open curved orbits and conformal structures on closed curved orbits,
  which are embedded hypersurfaces. In particular, if one considers
  either of the two cases on a compact manifold with boundary for
  which the closed curved orbit coincides with the boundary, one
  obtains a compactification of a (non--compact complete) Einstein
  manifold by adding a conformal structure at infinity. However, this
  leads to two different types of compactifications, see
  \cite{CGH-projective} for more details on this.
\end{remark*}

\subsection{A CR--analog}\label{CR-almost-Einstein}
We conclude this article by looking at a complex analog of almost
Einstein scales in the realm of CR geometry. Since this case is
significantly more complicated than the conformal one, we restrict to
elementary aspects of the description here, and we will take this
topic up in more detail elsewhere.  We discuss this example on the one
hand because it gives rise to a CR version of the Einstein condition,
which is of intrinsic interest. On the other hand there are strong
indications that it will lead to a notion of compactifying a
non--compact complete K\"ahler--Einstein manifold by adding an infinity
carrying a CR structure, which is different from the one discussed in
Section \ref{complex-projective}.

We have already briefly discussed the description of CR structures as
parabolic geometries in Section \ref{complex-projective}. The basic group here
is $G=SU(p+1,q+1)$ and the parabolic subgroup $P\subset G$ is the
stabiliser of an isotropic (complex) line in the standard
representation $\Bbb C^{p+q+2}$. Regular normal parabolic geometries of
this type turn out to be equivalent to partially integrable almost CR
structures of hypersurface type, which are non--degenerate of
signature $(p,q)$ together with the choice of a certain root of the
canonical bundle, compare with \cite{cap-gover-cr-tractors}. Forming the
associated bundle corresponding to the standard representation, one
obtains the \textit{standard tractor bundle}. This is a complex vector
bundle $\Cal T$ of rank $p+q+2$ endowed with a canonical Hermitian
metric $h$ of signature $(p+1,q+1)$, a complex line--subbundle with
isotropic fibres, and a canonical Hermitian connection. We want to
study holonomy reductions determined by a parallel section of the
standard tractor bundle. As in \ref{almost-Einstein}, for a parallel
section $s$ of this bundle the function $h(s,s)$ is constant, and up
to constant rescalings of $s$, the basic $G$--types are distinguished
by the sign of $h(s,s)$. We will only analyse the case $h(s,s)<0$
here, the case $h(s,s)>0$ is closely parallel, and these are the
cases related to compactifications as discussed above.

The homogeneous model $G/P$ is the space of isotropic lines in $\Bbb
C^{p+q+2}$, and the parallel standard tractor is determined by
a negative vector $v\in\Bbb C^{p+q+2}$.  The stabiliser $H$ of $v$
in $G$ is isomorphic to $SU(p+1,q)$ via the action on the
orthocomplement $v^\perp$. Evidently, there are two $H$--orbits in
$G/P$, consisting of the lines contained in $v^\perp$ and the lines
transversal to $v^\perp$, respectively. The latter orbit is clearly
open, while the former forms a smooth embedded submanifold of real
codimension two if $q>0$ and is empty if $q=0$. The normal solution
$\si$ corresponding to $v$ is obtained by interpreting the inner
product with $v$ as a homogeneous function on the null cone in
$\Bbb C^{p+q+2}$, and thus as a section of a (complex) density bundle
on the space $G/P$ of isotropic lines.  Consequently, the closed
$H$--orbit coincides with the zero set $\Cal Z(\si)$.

Via Theorem \ref{2.5} this description readily carries over to a
parallel standard tractor $s$ with $h(s,s)<0$ on general curved
geometries. The curved orbit decomposition has the form $M=M_+\cup
M_0$, where $M_+$ is open, and $M_0$ coincides with the zero set of
the underlying normal solution and, if non--empty, consists of smoothly
embedded submanifolds of real codimension two. Note that $M_0$ must be
empty if $q=0$.

Let us next describe the induced Cartan geometries. In the case of
$M_0$, the distinguished isotropic line is contained in $v^\perp$, so
we can simply identify $H\cap P$ with the stabiliser of a complex
isotropic line in the standard representation $\Bbb C^{p+q+1}$ of
$H\cong SU(p+1,q)$. It is straightforward to verify directly that this
induced Cartan geometry is automatically regular, thus giving rise to
a partially integrable almost CR--structure on $M_0$ of hypersurface
type, which is non--degenerate of signature $(p,q-1)$. 

In the case of $M_+$, we have to determine the stabiliser in $H$ of
an isotropic line $\ell$ which is transversal to $v^\perp$. Now
elementary linear algebra shows that there is a unique vector $w\in
v^\perp$ such that $v+w\in\ell$. Clearly, any element of $H$ which
stabilises $\ell$, also has to stabilise $w$, and the converse also
holds. Since $v$ is negative and $v+w$ is null, $w$ must be positive,
and we see that $H\cap P\cong SU(p,q)$. Passing to the Lie algebra
level, we get a similar matrix presentation to that in Section
\ref{complex-projective}:
$$
H\cap P=\left\{\begin{pmatrix} 0 & 0\\
    0 & A \end{pmatrix}\right\}\subset
\left\{\begin{pmatrix} ix & -Z^*\Bbb I\\
    Z & A-\tfrac{ix}{p+q}\id \end{pmatrix}\right\}=H,
$$
where $x\in\Bbb R$, $Z\in\Bbb C^{p+q}$, $A\in\frak{su}(p,q)$, and
$\Bbb I=\Bbb I_{p,q}$. It is easy to see that the component in $\frak
h/(\frak h\cap\frak p)$ determined by $Z$ corresponds the contact
distribution $HM_+$ (equipped with the complex structure), while
projecting to the component determined by $x$ gives rise to a
distinguished contact form. This also determines a Hermitian metric on
$HM_+$ which is extended to a Riemannian metric on $TM_+$. In
addition, the induced Cartan geometry of type $(H,H\cap P)$ on $M_+$
determines a linear connection on $TM_+$ which is compatible with all
these structures. However, the discussion of consequences of normality
is much more complicated here than in the real case, in particular if
the initial structure has torsion, and we will not go into this.

Compared to the real case, there is an entirely new feature here,
however. Namely, the parallel section $s\in\Ga(\Cal T)$ determines a
parallel section of the bundle $\frak{su}(\Cal T)$, which is given by
the tracefree part of $\tilde s\mapsto h(\tilde s,s)Js$. Now
$\frak{su}(\Cal T)$ is the adjoint tractor bundle associated to the CR
structure, so a parallel section of this bundle determines a normal
infinitesimal automorphism of the geometry. It is easy to see, that on
$M_+$, this infinitesimal automorphism is nowhere vanishing, and it
is even transversal to the contact distribution there. Hence the flow
lines of this infinitesimal automorphism determine a foliation of
$M_+$ with leaves of real dimension one, and one can form local spaces
of leaves. 

The component $M_0$ can be interpreted as a CR infinity for the local
leaf spaces. Moreover, since the leaves of the foliation are
transversal to the contact distribution, any tangent space of such a
leaf space can be identified with the contact subspaces along the
leaf. Then one gets an almost complex structure on the leaf
space. Using Theorem \ref{2.5}, these facts follow by comparison to the
homogeneous model.  Moreover, on each local leaf space one gets an induced
Cartan geometry of type $(SU(p+1,q),S(U(1)\x U(p,q)))$. 

As a final remark we note that this structure  can then
be analysed in a similar way to the structure in Section
\ref{complex-projective}.  If the initial Cartan geometry is torsion
free (i.e.~if the initial structure is CR), then away from $M_0$ this
is precisely the setting considered by \cite{leitner-pseudo} and so it
can be expected that in fact this recovers the K\"ahler--Einstein
metric found there.

\subsection*{Acknowledgements}

\noindent 
M.H. benefited from many interesting discussions on conformal holonomy
with Stuart Armstrong and Felipe Leitner at several `Srni Winter
School Geometry and Physics'-conferences, and also on other occasions.
\v{C}ap and Hammerl appreciate support through the project P23244-N13 of the
"Fonds zur F\"{o}rderung der wissenschaftlichen For\-schung"
(FWF). Hammerl was also supported by a Junior Research Fellowship of
The Erwin Schr\"{o}dinger International Institute for Mathematical
Physics (ESI).

\noindent \v Cap and Gover are thankful for the support of the Royal
Society of New Zealand (Marsden Grant 10-UOA-113).

\noindent The work on this article was completed during the workshop
``Cartan connections, geometry of homogeneous spaces, and dynamics" at
the ESI.  Support of the University of Vienna via the ESI is
gratefully acknowledged.

\def\polhk#1{\setbox0=\hbox{#1}{\ooalign{\hidewidth
  \lower1.5ex\hbox{`}\hidewidth\crcr\unhbox0}}} \def\cprime{$'$}

\end{document}